\documentclass[11pt,twoside]{amsart}

\usepackage[dvipsnames]{xcolor}
\usepackage{hyperref}
\usepackage{cleveref}
\hypersetup{colorlinks=true, linkcolor=Blue, citecolor=Maroon}
\usepackage{amsthm, amsmath, amscd, amssymb,centernot,txfonts}
\usepackage{tikz-cd}
\usepackage{lipsum}
\usepackage{multicol}
\usepackage[normalem]{ulem}
\usepackage{amsfonts}
\usepackage{adjustbox}
\usepackage{mathrsfs}
\usepackage[all]{xy}
\usepackage{geometry}
\usepackage{dirtytalk}
\usepackage{mathtools}
\usepackage{enumitem}
\usepackage{comment}
\usepackage{fourier} 
\usepackage{capt-of}
\usetikzlibrary{shapes,arrows}
\usepackage{todonotes}
\usepackage{cancel}
\DeclarePairedDelimiter\floor{\lfloor}{\rfloor}
\DeclareMathOperator{\LCliff}{LCliff}
\DeclareMathOperator{\RCliff}{RCliff}
\DeclareMathOperator{\gon}{gon}
\DeclareMathOperator{\Ext}{Ext}
\DeclareMathOperator{\Sec}{Sec}
\DeclareMathOperator{\Proj}{Proj}
\DeclareMathOperator{\Sym}{Sym}
\DeclareMathOperator{\Hom}{Hom}

\theoremstyle{plain}
\numberwithin{equation}{section}
\newtheorem{theorem}{Theorem}[section]
\newtheorem{proposition}[theorem]{Proposition}
\newtheorem{lemma}[theorem]{Lemma}
\newtheorem{corollary}[theorem]{Corollary}

\newtheorem{question}[theorem]{Question}

\theoremstyle{definition}
\newtheorem{remark}[theorem]{Remark}
\newtheorem{definition}[theorem]{Definition}

\newtheorem{example}[theorem]{Example}
\newtheorem{set-up}[theorem]{Set-up}

\newcommand*{\QEDA}{\hfill\ensuremath{\blacksquare}}

\tikzstyle{decision} = [diamond, draw, , 
    text width=4.5em, text badly centered, node distance=3cm, inner sep=0pt]
\tikzstyle{block} = [rectangle, draw, , 
    text width=10em, text centered, rounded corners, minimum height=2em]
\tikzstyle{block1} = [rectangle, draw, , 
    text width=5em, text centered, rounded corners, minimum height=2em]    
\tikzstyle{line} = [draw, -latex']
\tikzstyle{cloud} = [draw, ellipse,, node distance=3cm,
    minimum height=2em]
\definecolor{light-gray}{gray}{0.95}    
\begin{document}

\title[Syzygies of canonical ribbons on higher genus curves]{Syzygies of canonical ribbons on higher genus curves}

\author[A. Deopurkar]{Anand Deopurkar}
\address{Mathematical Sciences Institute, The Australian National University}
\email{anand.deopurkar@anu.edu.au}

\author[J. Mukherjee ]{Jayan Mukherjee} 
\address{Department of Mathematics, Oklahoma State University, Stillwater, USA}
\email{jayan.mukherjee@okstate.edu}


\maketitle
\begin{abstract}
We study the syzygies of the canonical embedding of a ribbon $\widetilde{C}$ on a curve $C$ of genus $g \geq 1$.  We show that the linear series Clifford index and the resolution Clifford index are equal for a general ribbon of arithmetic genus $p_a$ on a general curve of genus $g$ with $p_{a} \geq \operatorname{max}\{3g+7, 6g-4\}$. Among non-general ribbons, the case of split ribbons is particularly interesting.  
Equality of the two Clifford indices for a split ribbon is related to the gonality conjecture for $C$ and it implies Green's conjecture for all double covers $C'$ of $C$ with $g(C') \geq \textrm{max}\{3g+2, 6g-4\}$.  We reduce it to the vanishing of certain Koszul cohomology groups of an auxiliary module of syzygies associated to $C$, which may be of independent interest.

\end{abstract}

\section{Introduction} 



Let $X \subset \mathbb{P}^N$ be a projective variety.
Let $I \subset k[X_0, \dots, X_N] =: S$ be its homogeneous ideal, $R := S/I$ the homogeneous coordinate ring, and $L = \mathcal O_X(1)$.
A fundamental problem is to understand the connection between geometric properties of $(X, L)$ and homological/algebraic properties of $R$ as an $S$ module.
In the mid 1980's, Mark Green used the Koszul cohomology groups $K_{p,q}(X,L)$ to formulate such a connection \cite{G84,G84b}. 
Out of this work arose his conjecture about the Clifford index of a smooth projective curve $C$ and the shape of the minimal free resolution of the homogeneous coordinate ring of its canonical embedding \cite{G84,GL86}.
This conjecture states that the Clifford index of $C$, a geometric quantity defined using linear series on $C$, is equal to its resolution Clifford index, an algebraic quantity related to the vanishing of the Koszul cohomology groups.
In \cite{BE95} and \cite{EG95}, the authors laid out an approach to prove Green's conjecture for a general curve by proving it for an everywhere non-reduced curve called a ribbon.
We recall the definition.
\begin{definition}\label{ribbon}
  A ribbon $\widetilde{Y}$ on a reduced connected scheme $Y$ is a non-reduced scheme with $\widetilde{Y}_{\textrm{red}} = Y$ such that
   \begin{enumerate}
       \item the ideal sheaf $\mathcal{I}_{Y/\widetilde{Y}} \subset \mathcal O_{\widetilde Y}$ satisfies $\mathcal{I}_{Y/\widetilde{Y}}^2 = 0$ and 
       \item viewed as an $\mathcal O_{Y}$ module, $\mathcal{I}_{Y/\widetilde{Y}}$ is a locally free sheaf $L$ of rank one (called the conormal bundle of $\widetilde Y$).
    \end{enumerate}
\end{definition}
The canonical ribbon conjecture, formulated in \cite{BE95} and \cite{EG95}, states that the analogue of Green's conjecture holds for a rational ribbon, that is, a ribbon $\widetilde C$ with $\widetilde{C}_{\textrm{red}} \cong \mathbb{P}^1$.
By semicontinuity, it implies Green's conjecture for a general curve.
In a series of pioneering papers, Voisin proved Green's conjecture for a general curve using different methods 
\cite{V02}, \cite{V05}.
Using the results of Voisin and Hirschowitz--Ramanan \cite{HR98}, the first author proved the canonical ribbon conjecture for rational ribbons \cite{D18}.
Eventually, Raicu--Sam \cite{RS19} and Park \cite{P22} gave independent proofs of the canonical ribbon conjecture, without using the results of Voisin and Hirschowitz--Ramanan, obtaining a proof of Green's conjecture for general curves via ribbons.

In this context, we study the minimal free resolutions of canonical ribbons over higher genus curves. It has been shown in \cite{Gon06} and \cite{GGP08}, that in almost all cases such ribbons arise as flat limits of smooth curves.
We see many compelling reasons to take up this study.

First, the canonical ribbon conjecture for ribbons is interesting for the same reasons as it is for smooth curves.
We have a geometric notion of linear series on ribbons, formulated in \cite{EG95}.
By the same argument as for smooth curves, presence of linear series produces non-linear syzygies, or equivalently, non-vanishing $K_{p,2}$'s.
It is natural to ask whether there are other reasons for non-linear syzygies, whose origins we do not understand.
It is worthwhile to understand the answer in as many situations as we can.
Ribbons on higher genus curves is one such setting.

Second, by semi-continuity Green's conjecture for split ribbons on $C$ imply Green's conjecture for all double covers of $C$.
Thus, understanding the minimal free resolution of the split ribbon has immediate payoff.

Third, ribbons provide one of the few ways of producing higher genus curves from lower genus ones--a kind of ``induction''.
A ribbon $\widetilde C$ on a smooth curve $C$ is given purely in terms of linear algebraic data: a rank 2 bundle obtained as an extension of $\Omega_{C}$ by the conormal bundle.
Thus, ribbons allow us to connect questions about higher genus curves to questions about bundles on lower genus curves.
The geometry of bundles on higher genus curves is much richer than that on $\mathbb{P}^1$.
It is natural to ask if it leads to a richer connection.

Finally, the geography of the moduli space of ribbons over higher genus curves is much more interesting.
In genus $0$, the space of ribbons is naturally the ambient space of a rational normal curve.
This space has many different statifications: $(1)$ the stratification by the secant varieties of the rational normal curve, $(2)$ the stratification by the blow-up index of the ribbon, $(3)$ the stratification by the splitting type of the rank $2$ bundle defining the ribbon, and $(4)$ the stratification by the shape of the minimal free resolution of the canonically embedded ribbon.
A consequence of the canonical ribbon conjecture is that \emph{all} these stratifications agree.
All stratifications have analogues for ribbons on higher genus curves.
Understanding whether and how they differ is an intrinsically interesting question.

\subsection{Results} 
Our first result is the canonical ribbon conjecture for a general ribbon of large genus on a general curve
\begin{theorem} [\Cref{odd genus maximum blow-up index}, \Cref{generic green even}, \Cref{any genus and any blowup index} in the main text]
  \label{Green's conjecture intro}
  Let $C$ be a general curve of genus $g$.
  Fix a line bundle $L$ on $C$.
  Let $\widetilde{C}$ be a general ribbon with conormal bundle $L$.
  Let $p_{a}(\widetilde C)$ be the arithmetic genus, $\gon(\widetilde C)$ the gonality, $\LCliff(\widetilde C)$ the linear series Clifford index, and $\RCliff(\widetilde C)$ the resolution Clifford index of $\widetilde C$.
  If $p_{a}(\widetilde C) \geq \operatorname{max}\{3g+7, 6g-4\}$, then we have
 \[ \LCliff(\widetilde C) = \gon(\widetilde C)-2 = \RCliff(\widetilde C) = \left\lfloor \frac{p_a(\widetilde C) - 1}{2} \right \rfloor. \] 
\end{theorem}

To prove \Cref{Green's conjecture intro}, we study two stratifications on the space $\mathbb{P}(H^0(2K_C-L)^*)$ of ribbons on $C$ with conormal bundle $L$.
The first is the stratification by the blow-up index of the ribbon, and the second is the stratification by gonality.
These two stratifications coincide with the secant stratification in the case when $C \cong \mathbb{P}^1$.
In higher genus, the blow-up index stratification coincides with the secant stratification, but the gonality stratification is different and more interesting.
Let $\textrm{Sec}_k(C)$ denote the $k$th secant variety of $C$ in the embedding $C \hookrightarrow \mathbb{P}H^0(2K_C-L)^*$.
\begin{theorem}[\Cref{gonality stratification} and \Cref{inclusion of gonality stratification in secant variety}]
  \label{gonality stratification introduction}
  Let  $C$ be a smooth curve of genus $g$ and gonality $m$. Then the subvariety $W_d$ of $\mathbb{P}(H^0(2K_C-L)^*)$ parameterizing ribbons $\widetilde{C}$ with conormal bundle $L$ containing a $g^1_d$ is the subvariety of $\Sec_{d+2g-2}(C)$ given by
\begin{align*}
  W_d = \bigcup_{e \leq d/2} \left\{\text{\parbox{.7\linewidth}{Secant planes spanned by divisors of length $d+2g-2$ that contain a ramification divisor of a degree $e$ map $C \to \mathbb{P}^1$}}\right\}.
\end{align*}
In particular, we have \[\textrm{Sec}_{d-2m}(C) \subset W_d \subset \textrm{Sec}_{d+2g-2}(C).\]
\end{theorem}


We turn next to the split ribbon. This is a particularly interesting case---all smooth double covers isotrivially degenerate to a split ribbon, and hence, a proof of equality of the linear series and resolution Clifford index for a split ribbon establishes Green's conjecture for a smooth double cover. We first show that the linear series Clifford index of a split ribbon is $2m-2$, where $m$ is the gonality of the underlying curve (\Cref{lower bounds to linear series gonality and Clifford index}). The rest of the section is devoted to proving that the following statements are equivalent to the fact that the resolution Clifford index of a split ribbon is $2m-2$.  We prove the following.

\begin{theorem}\label{green for split intro} (see \Cref{green for split}, \Cref{green for split implies green for double cover})
Let $\widetilde{C}$ be a split ribbon on a curve $C$ of genus $g$ and gonality $m$ with conormal bundle $L$.
\begin{enumerate}
    \item The linear series Clifford index of $\widetilde C$ is $\LCliff(\widetilde C) = 2m-2$.
    \item If $p_a(\widetilde{C}) \geq \textrm{max}\{2g+2m-1, 6g-4\}$, then the following are equivalent
    \smallskip
    \begin{enumerate}
    \item The resolution Clifford index $\RCliff(\widetilde C) = 2m-2$.
        \smallskip
      \item For all $i,j$ with $i,j \geq 0$ and $i + j = 2m-3$, the map
        \[\Phi_{i,j}: \bigwedge^{i+1}H^0(K_C) \otimes K_{j,1}(C, K_C-L) \xrightarrow{} \bigwedge^{i}H^0(K_C) \otimes K_{j,1}(C, K_C, K_C-L)\] is surjective.

        \smallskip
      \item Set $M^j = \bigoplus_q K_{j,1}(C, qK_C, K_C-L)$.
        For all $i,j$ with $i,j \geq 0$ and $i + j = 2m-3$, we have \[K_{i,1}(M^{j}, H^0(K_C)) = 0.\]

       \smallskip

       
\end{enumerate}

       \smallskip

    \item 
    If any of the equivalent conditions of part $(2)$ holds, then any smooth curve $C'$ which is a double cover of $C$ branched along $-2L$ satisfies Green's conjecture.
\end{enumerate}  
\end{theorem}
We believe that it is interesting to study whether the equivalent statements are true for any smooth curve $C$ (see \Cref{split ribbon conjecture}). We show that the question has a positive answer when the curve is either elliptic or hyperelliptic. We end with \Cref{does not satisfy Green}, where we give an example of a split ribbon of low arithmetic genus, i.e, arithmetic genus $9$ over a general curve of genus $3$, which does \textit{not} satisfy Green's conjecture.
\subsection{Organization and conventions}
In \Cref{sec:can}, we recall basic results about the canonical embedding of a ribbon.
In \Cref{sec:cliff}, we recall and establish the basic properties of the Clifford index for ribbons.
In \Cref{sec:blowup} and \Cref{sec:gonality}, we study the blow-up index and gonality stratification of the space of ribbons.
In \Cref{sec:green_general}, we use these stratifications to prove Green's conjecture for a general ribbon.
In \Cref{sec:split}, we treat the case of a split ribbon.

All schemes are of finite type over an algebraically closed field of characteristic zero.
For us, the projectivisation of $V$ is $\mathbb{P} V = \Proj \Sym^{*} (V^{*})$, which is the space of one-dimensional subspaces of $V$.
Unless stated otherwise, a \emph{curve} is a projective connected scheme of pure dimension 1.

\subsection*{Acknowledgements}
The second author thanks Purnaprajna Bangere and Debaditya Raychaudhury for motivating discussions on syzygies of $K3$ carpets. The second author also thanks the Australian National University for its hospitality.

\section{Canonical embedding of ribbons on curves of higher genus}\label{sec:can}

In this section, we study the canonical map of a ribbon on a higher genus curve.
Let $C$ be a smooth curve, $L$ a line bundle on $C$, and $\widetilde{C}$ a ribbon on $C$ with conormal bundle $L$.

We begin by analyzing sections of positive line bundles on $\widetilde C$.
\begin{proposition}\label{very ampleness on ribbons}
In the setup above, let $\widetilde M$ be a line bundle on $\widetilde C$ and set $M := \widetilde{M}|_C$.
\begin{enumerate}
    \item Assume that $M$ is very ample, $M \otimes L$ is base point free and $H^0(\widetilde{M}) \to H^0(M)$ surjects. Then $\widetilde{M}$ is very ample. 
    \item Assume that  $M$ is projectively normal and for all $k \geq 1$, the maps \[H^0(k\widetilde{M}) \to H^0(kM) \text{ and } H^0(M \otimes L) \otimes H^0(kM) \to H^0((k+1)M \otimes L)\] are surjective. Then the multiplication map $H^0(k\widetilde{M}) \otimes H^0(\widetilde{M}) \to H^0((k+1)\widetilde{M}) $ is surjective for all $k \geq 1$.
\end{enumerate}
\end{proposition}
\begin{proof}
Assuming (1), we have an exact sequence,
\begin{equation*}
    0 \to H^0(M \otimes L) \to H^0(\widetilde{M}) \to H^0(M) \to 0
\end{equation*}
Let $\widetilde{\zeta}$ be a length two subscheme of $\widetilde{C}$.
We must show that the restriction $H^0(\widetilde M) \to H^0(\widetilde M|_{\widetilde \zeta})$ surjects.
Let $\zeta$ be the intersection of $\widetilde \zeta$ with $C$, so that $\mathcal O_\zeta = \mathcal O_{\widetilde \zeta} |_C$.
Then $\zeta$ is a subscheme of $C$ of length one or two.
If it is of length two, then $\widetilde \zeta$ is a subscheme of $C$.
Since $M$ is very ample, and $H^0(\widetilde{M}) \to H^0(M)$ surjects, we conclude that $H^{0}(\widetilde M) \to H^{0}(\widetilde M|_{\widetilde \zeta})$ surjects.
we get a section of $\widetilde{M}$ separating $\mathscr{O}_{\widetilde{\zeta}}$.
On the other hand, if $\mathscr{O}_{\zeta}$ is of length one, then we have the exact sequence
\begin{equation*}
    0 \to \mathscr{O}_{\zeta} \to \mathscr{O}_{\widetilde{\zeta}} \to \mathscr{O}_{\zeta} \to 0
\end{equation*}
Tensoring by $\widetilde{M}$ and taking global sections gives
\begin{equation*}
    0 \to H^0(M \otimes L \otimes \mathscr{O}_{\zeta}) \to H^0(\widetilde{M} \otimes \mathscr{O}_{\widetilde{\zeta}}) \to H^0(M \otimes \mathscr{O}_{\zeta}) \to 0,
\end{equation*}
and a commutative diagram of exact rows as follows
\[
\begin{tikzcd}
  0 \arrow[r] & H^0(M \otimes L ) \arrow[r] \arrow[d] & H^0(\widetilde{M}) \arrow[r] \arrow[d] & H^0(M) \arrow[r] \arrow[d] & 0 \\
 0 \arrow[r] & H^0(M \otimes L \otimes \mathscr{O}_{\zeta}) \arrow[r] & H^0(\widetilde{M} \otimes \mathscr{O}_{\widetilde{\zeta}}) \arrow[r]  & H^0(M \otimes \mathscr{O}_{\zeta}) \arrow[r] & 0.
\end{tikzcd}
\]
Since both $M$ and $M \otimes L$ are base point free, the flanking vertical maps surject and hence the middle vertical map surjects. 

Assuming (2), we need to show that $H^0(k\widetilde{M}) \otimes H^0(\widetilde{M}) \to H^0((k+1)\widetilde{M})$ surjects. We have the following commutative diagram 
\[
\begin{tikzcd}
  0 \arrow[r] & H^0(M \otimes L ) \otimes H^0(k\widetilde{M}) \arrow[r] \arrow[d] & H^0(\widetilde{M}) \otimes H^0(k\widetilde{M}) \arrow[r] \arrow[d] & H^0(M) \otimes H^0(k\widetilde{M}) \arrow[r] \arrow[d] & 0 \\
 0 \arrow[r] & H^0((k+1)M \otimes L)) \arrow[r] & H^0((k+1)\widetilde{M}) \arrow[r]  & H^0((k+1)M) \arrow[r] & 0
\end{tikzcd}
\]
The conditions in (2) imply that the two flanking vertical maps surject, and hence the middle vertical map also surjects.
\end{proof}

We now apply the previous analysis to the canonical bundle.
\begin{proposition}\label{canonical morphism of ribbons}
Let $\widetilde{C}$ be a ribbon on a smooth irreducible curve $C$ of genus $g \geq 1$ with conormal bundle $L$. 
\begin{enumerate}
    \item If $p_a \geq 2g+2$, then $ K_{\widetilde{C}} $ is very ample.
    \item If $p_a \geq 2g+2$ and either $g = 1$ or $h^0(K_C+L) \leq g-2$, 
     then the map 
    $$H^0(kK_{\widetilde{C}}) \otimes H^0(K_{\widetilde{C}}) \to H^0((k+1)K_{\widetilde{C}})$$ is surjective for all $k \geq 1$. In particular if $g \geq 2$, the above surjectivity holds if $p_a \geq 4g-2$.
\end{enumerate}
\end{proposition}

\begin{proof}
The canonical bundle $K_{\widetilde{C}}$ of a ribbon sits in an exact sequence
\begin{equation*}
    0 \to K_C \to K_{\widetilde{C}} \to K_{\widetilde{C}}|_C \to 0,
  \end{equation*}
  and we have $K_{\widetilde{C}}|_C = K_C \otimes L^{-1}$.
  For $k \geq 2$, the space $H^1(k(K_C \otimes L^{-1}) \otimes L)$ vanishes while for $k = 1$, the map $H^1(K_C) \to H^1(K_{\widetilde{C}})$ is injective. Hence $H^0(kK_{\widetilde{C}}) \to H^0(k(K_C \otimes L^{-1}))$ is surjective for all $k \geq 1$.
  Observe that $K_{C}$ is base point free and $K_C \otimes L^{-1}$ is very ample if $-\deg L \geq 3$.
  So (1) follows from \Cref{very ampleness on ribbons}.
  Further, under the conditions of (2) the map \[H^0(K_C) \otimes H^0(k(K_C\otimes L^{-1})) \to H^0((k+1)(K_C \otimes L^{-1}) \otimes L)\]
  is surjective by \cite[Theorem $4.e.1$]{G84}
  Hence part (2) follows from the part (2) of \Cref{very ampleness on ribbons}.
\end{proof}





\section{Linear Series Clifford index and resolution Clifford index of a ribbon}\label{sec:cliff}
Let $C$ be a smooth  curve of genus $g$ and let $\widetilde C$ be a ribbon on $C$ of arithmetic genus $p_a$ and conormal bundle $L$.
The exact sequence
\[ 0 \to L \to \mathcal{O}_{\widetilde C} \to O_C \to 0\]
implies
\[ \chi\left(\mathcal{O}_{\widetilde{C}}\right) = \chi(\mathcal{O}_C) + \chi(L),\]
and hence
\[ p_a = (2g-1) -\deg L.\]

We recall some defintions as introduced in \cite[Section~1]{EG95}.
A \emph{generalized line bundle} on $\widetilde{C}$ is a torsion free coherent sheaf which is generically free of rank $1$.
We define the \emph{degree} of a generalized line bundle $\widetilde{M}$ by
\[\deg(\widetilde{M}) = \chi(\widetilde{M})-\chi(\mathcal{O}_{\widetilde{C}}).\]
The restriction of $\widetilde M$ to $C$ may have torsion; let $\tau \subset \widetilde M|_{C}$ be the torsion subsheaf.
We set ${M} = (\widetilde{M}|_C/\tau)$.

Sections of $\widetilde M$ are of two kinds, those that yield a non-zero section of $M$ and those that do not.
The sections of the first kind define an injective map $\mathcal O_{\widetilde C} \to \widetilde M$.
The scheme-theoretic vanishing locus of such a section is a Cartier divisor on $\widetilde C$.
We call such sections \emph{Cartier} sections.
The sections of the second kind do not give an injective map $\mathcal O_{\widetilde C} \to \widetilde M$.
The entire reduced curve $C$ is contained in their scheme-theoretic zero locus.

Let \(\widetilde M\) be a generalised line bundle on \(\widetilde C\).
By \cite[Theorem~1.1]{EG95}, there exists a unique divisor \(\beta \subset C\) and a line bundle \(\widetilde M'\) on the blow up \(\widetilde C'\) of \(\widetilde C\) along \(\beta\) such that \(\widetilde M\) is the push-forward of \(\widetilde M'\).
Then \(M = \widetilde {M'}|_{C}\).

A \emph{generalized linear series} of rank $r$ and degree $d$, or simply a $g^r_d$, on $\widetilde{C}$ is pair $\Phi = (V, \widetilde{M})$ where $\widetilde{M}$ is a generalized line bundle and $V \subset H^0(\widetilde{M})$ is of dimension $r+1$, such that the restriction map $V \to H^0(M)$ is injective.
The \emph{Clifford index} of $\Phi$ is $d-2r$.
The \emph{linear series Clifford index} of $\widetilde{C}$, denoted by $\LCliff(\widetilde C)$, is the minimum of the Clifford indices of all generalized linear series $g^r_d$ such that $r \geq 1$ and $d \leq p_a-1$.
The \emph{gonality} of $\widetilde{C}$, denoted by $\gon(\widetilde{C})$, is the smallest $d$ such that there exists a $g^1_d$ on $\widetilde{C}$.

For a line bundle $\widetilde{H}$ on $\widetilde{C}$, we let $K_{p,q}(\widetilde{C}, \widetilde{H})$ be the Koszul cohomology group as defined in \cite{G84}.
For $\widetilde H = K_{\widetilde C}$, we know that $K_{p,q}(\widetilde C, K_{\widetilde C})$ is possibly non-zero only for $0 \leq p \leq p_a-2$ and $0 \leq q 3$.
Within this range, the group vanishes for $(p > 0, q = 0)$ and $(p < p_a-2, q = 3)$.
So the most interesting cases are $q = 1$ and $q = 2$.
Owing to the duality,
\[ K_{p,q}(\widetilde C, K_{\widetilde C}) = K_{p_a-p-2,3-q}(\widetilde C, K_{\widetilde C})^{\vee},\]
understanding $K_{p,1}$'s is equivalent to understanding $K_{p,2}$'s.
We also know that if $K_{p,2}(\widetilde C, K_{\widetilde C}) = 0$ then for all $i \geq p$, we have $K_{i,2}(\widetilde C, K_{\widetilde C}) = 0$.
So it is important to understand the smallest $p$ such that $K_{p,2}(\widetilde C, K_{\widetilde C}) \neq 0$.
This $p$ is called the \emph{resolution Clifford index} of $\widetilde C$.
We denote it by $\RCliff(\widetilde C)$.
By duality, it is the smallest $p$ such that $K_{p_a-2-p,1}(\widetilde{C}, K_{\widetilde{C}}) \neq 0$.

\subsection{Semicontinuity of gonality and Clifford index}
The resolution Clifford index is lower semi-continuous by the semi-continuity of cohomology.
We now establish the lower semicontinuity of gonality and linear series Clifford index.

Fix a smooth curve $C$ of genus $g$ and a ribbon $\widetilde C$ on $C$ of arithmetic genus $p_a$.
Let $(T,0)$ be a smooth pointed curve and $\mathcal C \to T$ a flat proper morphism of relative dimension 1.
Suppose for all $t \in T$ with $t \neq 0$, the fiber $\mathcal C_t$ is a smooth  curve of genus $p_a$ and the fiber $\mathcal C_0$ is $\widetilde C$.

The following proposition constructs a limiting generalized $g^r_d$.
\begin{proposition}\label{limit grd}
  In the setup above, assume that for all $t \neq 0$, the curve $\mathcal C_t$ has a $g^{r}_{d}$.
  Assume that $p_a > d + 2g - 1$ and $r \geq 1$.
  Then $\widetilde C$ has a generalised $g^r_{d'}$ with $d' \leq d$.
\end{proposition}
\begin{proof}
  We follow the proof of \cite[Theorem~2.1]{EG95}, replacing \(\omega_{\mathcal C/T}\) by a more suitable line bundle.

  Let \(\eta\) be the generic point of \(T\).
  After a finite base change, we may assume that we have a line bundle \(\mathcal G_{\eta}\) on \(\mathcal C_{\eta}\) and an $(r+1)$-dimensional subspace \(V_{\eta} \subset H^0(\mathcal C_{\eta})\).
  Choose a line bundle \(\mathcal E\) on \(\mathcal C\) of relative degree \(2g-4\), for example, by starting with a line bundle of degree \(2g-4\) on \(\widetilde C\) and deforming it.
  Then $\mathcal E|_{C}$ is a line bundle of degree $g-2$.
  Assume that it is a general line bundle of this degree.
  
  Set \(\mathcal F = \omega_{\mathcal C/T} \otimes \mathcal E^{-1}\).
  \begin{align*}
    \deg\left(\mathcal F_{\eta} \otimes \mathcal G_{\eta}^{-1}\right) &= (2p_a-2)-(2g-4)-d \\
                                                                            &\geq p_a.
  \end{align*}
  The last inequality follows from our assumption $p_a > d + (2g-1)$.
  In particular, \(\mathcal F_{\eta} \otimes \mathcal G_{\eta}^{-1}\) is effective.
  One of its section gives an inclusion \(\mathcal G_{\eta} \subset \mathcal F_{\eta}\).
  Via this inclusion, we may think of \(V_{\eta}\) as a subspace of \(H^0(\mathcal F_{\eta})\).

  By the theorem on cohomology and base change \cite[Section~5]{mum}, there exists a map \(K^0 \to K^1\) of locally free \(\mathcal O_T\) modules of finite rank such that for every \(S \to T\), we have a canonical isomorphism
  \[ \pi_{*} (\mathcal C_T \times S, \mathcal F \times S) = \ker(K^0_S \to K^1_{S}).\]
  We have the subspace \(V_{\eta} \subset K^0_{\eta}\) of dimension $r+1$.
  It extends to a locally free $\mathcal O_{T}$ module of the same rank \(V \subset K^{0}\) such that the map \(V|_0 \to K^0|_0\) remains injective.
  Since \[V_{\eta} \subset \ker(K^0_{\eta} \to K^1_{\eta}) = H^0(\mathcal F_{\eta}),\] we have \[V \subset \ker (K^0 \to K^1) = \pi_{*}(\mathcal F)\]
  and \[V|_0 \subset \ker(K^0|_0 \to K^1|_0) = H^0(\mathcal F|_0)\]
  by continuity.
  
  We have the exact sequence
  \[ 0 \to K_C \to K_{\widetilde C} \to K_{\widetilde C}|_C \to 0.\]
  Tensoring by \(\mathcal E^{-1}\) yields the exact sequence
  \begin{equation}
    0 \to K_C \otimes \mathcal E^{-1} \to \mathcal F|_{0} \to \mathcal F|_{C} \to 0.
  \end{equation}
  We have \[\deg (K_{C} \otimes \mathcal  E^{-1})= (2g-2)-(g-2) = g.\]
  Since $\mathcal E|_{C}$ is general, so is $K_C \otimes \mathcal E^{-1}$, and hence $h^0(K_C \otimes \mathcal E^{-1}|_C) = 1$.
  By the long exact sequence, we see that
  \begin{equation}\label{eqn:negative kernel}
   \dim \ker \left( H^0(\mathcal F|_0) \to H^0(\mathcal F|_C) \right) = 1.
  \end{equation}

  Let \(\mathcal G\) be the subsheaf of \(\mathcal F\) generated by \(V\) and let \(G_0 \subset \mathcal F|_0\) be the image of \(\mathcal G\).
  Equivalently, \(G_0\) is the subsheaf of \(\mathcal F|_0\) generated by \(V_0\).
  Then $G_0$ is torsion free.
  Since $\dim V_0 = (r+1)$ is greater than the dimension $1$ of the kernel of $H^0(\mathcal F|_0) \to H^0(\mathcal F|_C)$, the restriction map $V_0 \to H^0(\mathcal F|_C)$ is non-zero.
  Then it follows that the map $G_0 \subset \mathcal F|_0$ is an isomorphism at the generic point.
  Therefore, $G_0$ is a generalized line bundle.
  Since we have a surjection $\mathcal G|_0 \to G_{0}$, and $\deg \mathcal G|_0 = d$, we conclude that $\deg G_0 \leq d$.
  
  Set \(G = G_0|_C / \textrm{torsion}\).
  We know that the map $V_0 \to H^0(G)$ is non-zero, so $\deg G \geq 0$.
  It remains to check that the map $V_0 \to H^0(G)$ is injective.
  Let $\beta \subset C$ be the divisor such that $G_0$ is the push-forward of a line bundle from the blow-up of $\widetilde C$ along $\beta$.
  Let $L$ be the conormal bundle of $\widetilde C$.
  Then we have the exact sequence
  \[ 0 \to G \otimes L(\beta) \to G_{0} \to G \to 0;\]
  see \cite[\S~1]{EG95}.
  Let $d' = \deg G_0$.
  Then \[\deg G + \deg \beta \leq 2\deg G + \deg \beta = d',\] so
  \[ \deg (G \otimes L(\beta)) \leq d' + \deg L =  d' +(2g-1) - p_a \leq d + (2g-1) - p_a < 0,\]
  where the last inequality follows from $p_a > d + (2g-1)$.
  As a result, $H^0(G \otimes L(\beta)) = 0$, and hence the map $H^0(G_0) \to H^0(G)$ is injective.
  Therefore, the composite $V_0 \subset H^0(G_0) \to H^0(G)$ is injective.
\end{proof}

\begin{corollary}\label{semicontinuity of gonality}
  In the setup of \Cref{limit grd}, suppose the generic fibers $\mathcal C_t$ have gonality $d$.
  If $p_a > d+2g-1$, then $\mathcal C$ has gonality at most $d$.
  In particular, if $p_a > 4g+1$, then $\mathcal C$ has gonality at most $d$.
\end{corollary}
\begin{proof}
  The first statement follows directly from Proposition\ref{limit grd}.
  For the second statement, observe that $d \leq (p_a+3)/2$.
  So $p_a > 4g+1$ implies $p_a > d+2g-1$.
\end{proof}
In \Cref{semicontinuity of gonality}, the condition \(p_a > d+2g-1\) is indeed necessary.
See \Cref{ex:ellipticg13} for the failure of the existence of a limiting \(g^1_3\) without this condition.

\begin{corollary}\label{semicontinuity of Clifford index and Clifford dimension}
  In the setup of \Cref{limit grd}, suppose the generic fibers $\mathcal C_t$ have Clifford index $c$ and Clifford dimension $r$.
  If $p_a > c+2r + (2g-1)$, then $\LCliff(\widetilde C) \leq c$.
  In particular, if $p_a > 4g-3 + 4r$, then $\LCliff(\widetilde C) \leq c$.
\end{corollary}
\begin{proof}
  The first statement follows directly from Proposition\ref{limit grd}.
  For the second statement, observe that $c \leq (p_a-1)/2$.
  So $p_a > 4g-3 + 4r$ implies $p_a > c+2r+(2g-1)$.
\end{proof}

Using the results of \cite{ELMS89}, we eliminate the dependence of $p_a$ on the Clifford dimension $r$ in Corollary~\ref{semicontinuity of Clifford index and Clifford dimension}.
\begin{corollary}\label{semicontinuity of Clifford index}
  In the setup of \Cref{limit grd}, assume that the fibers $\mathcal C_t$ have Clifford index $c$ and Clifford dimension $r$.
  Then $\LCliff(\widetilde C) \leq c$ holds under any of the following hypotheses:
\begin{enumerate}
\item $r = 1$ and $p_a > 4g+1$,
\item $p_a$ is odd and $p_a > 8g+1$,
\item $p_a$ is even, \((p_{a}, c, d) \neq (4r-2, 2r-3,4r-3)\), and $p_a > 8g+1$. 
\end{enumerate}
\end{corollary}
\begin{proof}
  Under the first hypothesis, the conclusion is Corollary\ref{semicontinuity of gonality}.
  Under the second or third hypothesis, \cite[Corollary~3.5]{ELMS89} says that $p_a\geq 8r-7$.
  Then $p_a > 8g+1$ implies $p_a > 4g-3 + 4r$.
  Therefore, the conclusion follows from Corollary~\ref{semicontinuity of Clifford index and Clifford dimension}.
\end{proof}

\subsection{Green-Lazarsfeld non-vanishing theorem for ribbons}
In this section, we relate the linear series Clifford index and the resolution Clifford index of a ribbon.
We show that, for $p_a$ large compared to $g$, we have the inequality $\RCliff \leq \LCliff$.

\begin{theorem}\label{GL vanishing 1}
  Let $\widetilde{C}$ be a ribbon of arithmetic genus $p_a$ on smooth  curve $C$ of genus $g$ with $h^0(\mathcal{O}_{\widetilde C}) = 1$.
  Let $\widetilde{M}_1$ and $\widetilde{M}_2$ be line bundles on $\widetilde C$ and set \[\widetilde{M} = \widetilde{M}_1 \otimes \widetilde{M}_2.\]
  For $i = 1,2$, let $H^0(\widetilde{M}_i)$ be of dimension $r_i+1$, with $r_i \geq 1$.
  Assume that
  \begin{enumerate}
  \item $\widetilde{M}_1$ is base point free, and
   \item the zero locus of a general element of $H^0(\widetilde{M}_2)$ is zero dimensional.
\end{enumerate}
Then
\begin{equation*}
    K_{r_1+r_2-1,1}(\widetilde{C}, \widetilde{M}) \neq 0
\end{equation*}
\end{theorem}

\begin{proof}
  Let $s_1 \in H^0(\widetilde M_2)$ be a section whose scheme-theoretic zero locus $D_1$ is zero dimensional.
  Let $s_2 \in H^0(\widetilde M_1)$ be a section whose scheme-theoretic zero locus $D_2$ is disjoint from $D_1$.
  Note that both $D_1$ and $D_2$ are Cartier diviors on $\widetilde C$.
  Then, up to scaling, there is a unique section $s_0$ of $H^0(\widetilde{M})$ that vanishes on both $D_1$ and $D_2$;
  it is the section whose scheme-theoretic zero locus is $D_1+D_2$.

  The rest of the proof follows verbatim from \cite[Appendix]{G84}.
  We sketch it for the convenience of the reader.
  Recall that a Cartier section of a line bundle on $\widetilde C$ is a section that defines an injection from $\mathcal{O}_{\mathcal{C}}$.
  The scheme theoretic zero locus of a Cartier section is a Cartier divisor.
  Since $s \in H^0(\widetilde{M})$ is Cartier, a general element of $H^0(\widetilde M)$ is Cartier.
  Consider $H^0(\widetilde{M}(-D_1))$ (resp. $H^0(\widetilde{M}(-D_2))$) seen as subspaces of $H^0(\widetilde{M})$.
  These two subspaces intersect along the one-dimensional subspace spanned by $s$.
  These subspaces contain Cartier sections $s_2$ and $s_1$, respectively, so their general section is Cartier.
  Choose bases of these subspaces consisting of Cartier sections as follows.
  Let \[s_0, s_1, ..., s_{r_2}\] be a basis of $H^0(\widetilde{M}(-D_2))$, and 
  \[s_0, s_{r-r_1+1},...,s_r\]
  a basis of $H^0(\widetilde{M}(-D_1))$.
  Extend to a basis of $H^0(\widetilde M)$ by adding Cartier sections
  \[s_{r_2+1},...,s_{r-r_1}.\]
Let $\{e_0, \dots, e_{r}\}$ be the dual basis of $H^0(\widetilde{M})^*$.
Let $\iota = \sum_{i=1}^{r-r_1} e_i \otimes s_i$ and $s = \sum_{i=0}^{r} e_i \otimes s_i $. Let $$ \alpha = \iota \wedge e_{r_2+1} \wedge...\wedge e_{r-r_1} $$
Then $\alpha \in \bigwedge^{r-r_1-r_2+1} H^0(\widetilde{M})^* \otimes H^0(\widetilde{M}(-D_2))$.
Consider $s \wedge \alpha$.
We have that
\begin{align*}
  s \wedge \alpha &\in \bigwedge^{r-r_1-r_2+2} H^0(\widetilde{M})^* \otimes H^0(\widetilde{M}(-D_2)\otimes\widetilde{M}(-D_1)) \\
                  & = \bigwedge^{r-r_1-r_2+2} H^0(\widetilde{M})^* \otimes H^0(\widetilde{M}) \\
  &= \bigwedge^{r_1+r_2-1} H^0(\widetilde{M}) \otimes H^0(\widetilde{M}).
\end{align*}
Since $s \wedge s \wedge \alpha = 0$, we see that $s$ defines a Koszul cocycle, that is, an element of $K_{r_1+r_2-1,1}(\widetilde{C}, \widetilde{M})$.
The fact that this element is non-zero follows exactly as in \cite[Appendix]{G84}.
\end{proof}


We now examine when line bundles residual to the canonical carry Cartier sections.
\begin{lemma}\label{Cartier divisor in residual}
  Let $\widetilde{C}$ be a ribbon of arithmetic genus $p_a$ on a smooth  curve $C$ of genus $g$.
  Let $\widetilde{M}_1$ be a line bundle on $\widetilde{C}$ with $h^0(\widetilde{M}_1) \geq r+1$ and $r \geq 1$.
  Set $d = \deg \widetilde M_1$ and $c = d-2r$.
  Let $\widetilde M_2 = K_{\widetilde C} \otimes \widetilde M_1^{-1}$.
  Assume that $d \leq p_a-1$ and $H^{0}(\widetilde M_1)$ contains a Cartier section.
  If $p_a > 3g-2+c$ then $h^0(\widetilde M_2) \geq r+1$ and $H^0(\widetilde{M}_2)$ contains a Cartier section.
\end{lemma}
\begin{proof}
  By Riemann--Roch, we have
  \[ h^{0}(\widetilde M_2) = p_a-d - 1 + h^0(\widetilde M_1).\]
  Since $d \leq p_a-1$ and $h^0(\widetilde M_1)\geq r+1$, we have $h^0(\widetilde M_2) \geq r+1$.
  
  Set $\widetilde{M}_2|_C = M_2$.
  Let the conormal bundle of $\widetilde{C}$ be $L$.
  We have the exact sequence
 \[0 \to M_2 \otimes L \to \widetilde{M}_2 \to M_2 \to 0,\]
 and hence
\begin{equation*}
    \deg(M_2 \otimes L) = \displaystyle\frac{1}{2}\deg(\widetilde{M}_2)+\deg(L) = \displaystyle\frac{1}{2}(2p_a-2-d)-(p_a-2g+1) = 2g-2-\displaystyle\frac{d}{2}.
\end{equation*}
If $d > 4g-4$, then every section of $H^0(\widetilde M_2)$ is Cartier.
Assume that $d \leq 4g-4$.
Then \[r = \displaystyle\frac{1}{2}(d-c) \leq 2g-2.\]
Multiplication by a Cartier section of $\widetilde M_1$ gives an injection \[\widetilde M_2 \hookrightarrow K_{\widetilde C}.\]
Recall the short exact sequence
\[0 \to H^0(K_C) \to H^0(K_{\widetilde{C}}) \to H^0(K_C \otimes L^{-1}).\]
The inclusion $\widetilde M_2 \hookrightarrow K_{\widetilde C}$ induces an injection \[H^0(M_2 \otimes L) \to H^0(K_C).\]
In particular, $h^0(M_2 \otimes L) \leq g$.
On the other hand, we know that \[h^0(\widetilde M_2) = p_a - d -1+ h^0(\widetilde M_1) \geq p_a - d + r.\]
From $p_a > 3g-2 + c$ and $d \leq 4g-4$, it follows that $p_a-d+r > g$ and hence $h^0(\widetilde M_2) > h^0(M_2 \otimes L)$.
So there exists a Cartier section of $\widetilde M_2$.
\end{proof}

We combine \Cref{GL vanishing 1} and \Cref{Cartier divisor in residual} to obtain the following non-vanishing result.
\begin{theorem}\label{lin and res}
  Let $\widetilde{C}$ be a ribbon of arithmetic genus $p_a$ on a smooth curve $C$ of genus $g$ and linear series Clifford index $\LCliff(\widetilde C)$.
  If $p_a > 3g-2+\LCliff(\widetilde C)$, then 
  $\RCliff(\widetilde C) \leq \LCliff(\widetilde C)$, that is,
  \[ K_{p_a-2-\LCliff(\widetilde C),1}(\widetilde{C}, K_{\widetilde{C}}) \neq 0.\]
  In particular, the non-vanishing holds if $p_a > 6g-5$.
\end{theorem}

\begin{proof}
  Let $L$ be the conormal bundle of $\widetilde{C}$.
  Since $p_a > 2g-1$, we have $\deg L = -(p_a-2g+1) < 0$, so $h^0(\mathcal{O}_{\widetilde{C}}) = 1$.

  Let $g^r_d$ be a generalized linear series on $\widetilde{C}$ with $d \leq p_a-1$ and $r \geq 1$ and $d-2r = c$.
  Recall that this is a pair $(V, \widetilde M_1)$ where $\widetilde{M}_1$ is a generalized line bundle of degree $d$ and $V \subset H^0(\widetilde{M}_1)$ is an $(r+1)$ dimensional vector subspace of global sections which injects into $H^0(M_1)$.
  We may assume that $\widetilde{M}_1$ is globally generated by $V$.
  Otherwise, we simply replace $\widetilde{M}_1$ by the subsheaf generated by $V$ (see \cite[Lemma~1.3]{EG95}).

  By \cite[Theorem~1.1]{EG95}, there exists a divisor $\beta \subset C$ of degree $b$, such that after blowing up $\widetilde{C}$ along $\beta$ we have a line bundle $\widetilde{M}_1'$ of degree $d' = d-b$ whose pushforward is $\widetilde{M}_1$. The blown up ribbon $\widetilde{C}'$ has conormal bundle $L(\beta)$ and arithmetic genus $p_a' = p_a-b$.
  Since $H^0(\widetilde{M}_1') = H^0(\widetilde{M}_1)$, we may treat $V$ as a subspace of $H^0(\widetilde{M}'_1)$.
  Since $d \leq p_a-1$, we also have $d' \leq p_a'-1$.
  Since $H^0(\widetilde{M}_1')$ contains $V$ of dimension $r+1$, we have $h^0(\widetilde M_1') \geq r+1$.
  Since $V \subset H^0(\widetilde M_1)$ generates $\widetilde M_1$, it follows that $V \subset H^0(\widetilde M_1')$ generates $\widetilde M_1'$.

  Set $c' = d'-2r = c-b$.
  Since $p_a > 3g-2+c$, we have $p_a' > 3g-2+c'$.
  Set $\widetilde M_2' = K_{\widetilde C'} \otimes (\widetilde{M}_1')^{-1}$. 
  By \Cref{Cartier divisor in residual}, $H^0(\widetilde M_2')$ contains a Cartier section.
  Write $h^0(\widetilde{M}'_1) = r+1+a$ for some $a \geq 0$.
  By Riemann-Roch, we have $h^0(\widetilde{M}_2') = p_a'-d'+r+a$.
  By \Cref{GL vanishing 1}, we get
  \begin{equation}\label{van 1}
    K_{p_a'-d'+2r-2+2a, 1}(\widetilde{C}', K_{\widetilde{C}'}) \neq 0.
  \end{equation}
  Since $a \geq 0$, we have \[p'_a - d'+2r+2a \geq p'_a-d'+2r-2 = p'_a-c'-2 = p_a-c-2.\]
  So from \eqref{van 1}, we conclude
  \[K_{p_a-c-2, 1}(\widetilde{C}', K_{\widetilde{C}'}) \neq 0.\]
   Then by \cite[Lemma~1]{D18}, we deduce
  \[K_{p_a-c-2, 1}(\widetilde{C}, K_{\widetilde{C}}) \neq 0.\]
  The proof is now complete.
\end{proof}

\section{Blow-up index stratification of the space of ribbons}\label{sec:blowup}
Recall that a \emph{split} ribbon $\widetilde C$ on $C$ is one that admits a retraction map $\widetilde C \to C$.
Every ribbon admits a blow-up that is a split ribbon.
The \emph{blow up index} of a ribbon is the minimum number of simple blow-ups necessary to make the ribbon split.
In this section, we study the set of ribbons of a given blow-up index, its relationship with a secant variety, and with the linear series Clifford index.

\subsection{Blow-up index of a ribbon as a pushout and relations to secant variety}
Fix a smooth  curve $C$.
A ribbon $\widetilde{C}$ on $C$ with conormal bundle $L$ gives an extension
\[ 0 \to L \to \Omega_{\widetilde{C}}|_C \to K_C \to 0.\]
Conversely, given an extension
\[ 0 \to L \to E \to K_C \to 0,\]
there is a unique ribbon $\widetilde C$ on $C$ with an isomorphism $E \cong \Omega_{\widetilde C}|_{C}$ compatible with the extension.
Therefore, the space of ribbons on $C$ with conormal bundle $L$ is identified with $\mathbb{P} \Ext^1(\Omega_C, L)$ (see \cite[Theorem~1.2]{BE95}).

Fix a ribbon $\widetilde C$.
Given a divisor $\beta \subset C$, we have the map $L \to L(\beta)$. 
Construct the push-out diagram
\[
\begin{tikzcd}
    0 \arrow[r] & L \arrow[r] \arrow[d] & \Omega_{\widetilde{C}}|_C \arrow[r] \arrow[d] & K_C \arrow[r] \arrow[equal,d] & 0 \\
    0 \arrow[r] & L(\beta) \arrow[r] & E \arrow[r] & K_C \arrow[r] & 0.
  \end{tikzcd}
\]
The extension in the second row corresponds to a ribbon $\widetilde{C}'$ with conormal bundle $L(\beta)$.
By \cite[Theorem~1.9]{BE95}, the ribbon $\widetilde{C}'$ is precisely the blow-up of $\widetilde{C}$ along the Weil divisor $\beta$ of $\widetilde{C}$.
Let $e$ be the class of the extension in the first row and $e'$ the extension in the second row.
Let
\begin{equation}\label{pushout}
  \mathbb{P}\Ext^1(K_{C}, L) \to \mathbb{P} \Ext^1(K_C, L(\beta))
\end{equation}
be the map induced by $L \to L(\beta)$.
Then $e'$ is the image of $e$.
If $\beta$ is of sufficiently large degree, then the group on the right vanishes, and hence $e'$ vanishes.
Then $\widetilde C'$ is the split ribbon.

\begin{definition}\label{blow-up index}
  Let $\widetilde{C}$ be a ribbon on a smooth projective curve $C$.
  The \emph{blow-up index} $b(\widetilde{C})$ of $\widetilde C$ is the smallest $b$ such that there exists a divisor $\beta$ on $C$ of degree $b$ such that the blow-up of $\widetilde C$ along $\beta$ is split.
\end{definition}

Let $e \in \mathbb P \Ext^1(K_C, L)$ be the extension class corresponding to $\widetilde C$.
Then $b(\widetilde C)$ is the smallest such that there exists an effective divisor $\beta \subset C$ such that the image of $e$ in $\mathbb P \Ext^1(K_C, L(\beta))$ vanishes.

Recall that $C$ is a smooth  curve.
By Serre duality, we have the identification
\[ \Ext^1(\Omega_C, L) = H^1(L\otimes K_C^{-1}) = H^0(K_C^{2} \otimes L^{-1})^{*}.\]
\begin{proposition}\label{blow-up index and secant variety}
  Let $L$ be a line bundle of negative degree on $C$.
  Let $i \colon C \hookrightarrow \mathbb{P}H^0(K^{2}_C\otimes L^{-1})^*$ be the embedding given by the complete linear series of the very ample line bundle $K_C^2 \otimes L^{-1}$.
  Then the ribbons $\widetilde{C}$ on $C$ with conormal bundle $L$ and $b(\widetilde{C}) \leq k$ correspond to the points of the $k$-secant variety of the embedding $i$. 
\end{proposition}

\begin{proof}
  Let $\widetilde{C}$ be a ribbon on $C$ with conormal bundle $L$ and extension class
  \[e \in H^0(K_C^{2}\otimes L^{-1})^{*} = H^{1}(L \otimes K_C^{-1}).\]
  Let $\beta \subset C$ be an effective divisor.
  Consider the map
\[
    H^1(L \otimes K_C^{-1}) \xrightarrow{f} H^1(L(\beta) \otimes K^{-1}_C).
  \]
The map above is Serre dual to the map
\[
  H^0(K^{2}_C\otimes L^{-1})^* \to H^0(K^{2}_C \otimes L^{-1}(-\beta))^*,
\]
which in turn is dual to the map
\[
  H^0(K_C^2 \otimes L^{-1}(-\beta)) \xrightarrow{g} H^{0}(K_C^2\otimes L^{-1})
\]
induced by the inclusion $\mathcal O(-\beta) \to \mathcal O$.
Observe that $e$ lies in the kernel of $f$ if and only if the composite
\[  H^0(K_C^2 \otimes L^{-1}(-\beta)) \xrightarrow{g} H^{0}(K_C^2\otimes L^{-1}) \xrightarrow{e} \mathbb k\]
vanishes.
But the points $\lambda \in \mathbb P H^0(K^2 \otimes L^{-1})$ such that the composite $\lambda \circ g$ vanishes are precisely the points that lie on the span of $\beta$ in the embedding $i$.
It follows that the ribbons $\widetilde C$ of blow-up index at most $k$ correspond the points lying on the span of an effective divisor of degree at most $k$, which is the $k$-secant variety of $C$.
\end{proof}

\begin{corollary}\label{Using secant varieties to compute blow-up index}
  Let \(C\) be a smooth curve of genus \(g\).
  Fix \(p_a > 2g-1\) and let \(L\) be a line bundle on \(C\) of degree \(-p_a+2g-1\).
  Let $\widetilde{C}$ be a ribbon of arithmetic genus with conormal bundle \(L\).
  Then
\begin{enumerate}
    \item $0 \leq b(\widetilde{C}) \leq \lceil (p_a+g-2)/2 \rceil$
    \item If $\widetilde{C} \in \mathbb P \Ext^1(K_C, L)$ is general, then $b(\widetilde{C}) = \lceil (p_a+g-2)/2 \rceil$ 
\end{enumerate}
\end{corollary}
\begin{proof}
  The $k-$th secant variety of a curve $C \hookrightarrow \mathbb{P}^N$ is of expected dimension $\min(2k-1, N)$ (see \cite{Lan84}).
  So the result follows from \Cref{blow-up index and secant variety}.
\end{proof}

\subsection{Blow-up index as pullback and relations to stability of $\Omega_{\widetilde{C}|_C}$}
There is another interpretation of the blow-up index, which is useful for measuring gonality.
As usual, let $\widetilde{C}$ be a ribbon on $C$ with conormal bundle $L$ defined by an extension $$ 0 \to L \to \Omega_{\widetilde{C}}|_C \to K_C \to 0.$$
An effective divisor \(\beta \subset C\) gives a map of line bundle \(K_C(-\beta) \to K_{C}\).
This map yields the pull-back diagram
\begin{equation}\label{pullback}
\begin{tikzcd}
    0 \arrow[r] & L \arrow[r] \arrow[d] & E \arrow[r] \arrow[d] & K_C(-\beta) \arrow[r] \arrow[d] & 0 \\
    0 \arrow[r] & L \arrow[r]  & \Omega_{\widetilde{C}}|_C \arrow[r]  & K_C \arrow[r]  & 0.
\end{tikzcd}
\end{equation}
Denoting by \(e \in \Ext^1(K_C,L)\) the class of the extension in the bottom row, the class of the extension in the top row is simply the image of \(e\) under the map
\begin{equation}\label{eq:pullback}
  \Ext^1(K_C, L) \to \Ext^1(K_C(-\beta), L).
\end{equation}

The observation leads to the following.
\begin{proposition}\label{equivalent definitions of blow-up index}
  Let \(C\) be a smooth curve and \(L\) a line bundle on \(C\).
  Let \(\widetilde C\) be a ribbon on \(C\) of arithmetic genus with conormal bundle \(L\).
  Let \(e \in \Ext^1(K_C,L)\) be the extension class of
  \[ 0 \to L \to \Omega_{\widetilde C}|_C \to K_C \to 0.\]
  The blow-up index of \(\widetilde C\) is any of the following equal quantities:
  \begin{enumerate}
  \item the minimum \(b\) such that there exists an effective divisor \(\beta\) on \(C\) of degree \(b\) such that the image of \(e\) under the push-out by \(L \to L(\beta)\) vanishes,
  \item the minimum \(b\) such that there exists an effective divisor \(\beta\) on \(C\) of degree \(b\) such that the image of \(e\) under the pull-back by \(K_C(-\beta) \to K_{C}\) vanishes.
  \end{enumerate}
  If $p_a > g+1$, then it is also equal to \(2g-2-k\) where \(k\) is the maximum such that \(\Omega_{\widetilde C}|_C\) has a sub line-bundle of degree \(k\).
\end{proposition}
\begin{proof}
  The quantity (1) is the definition of the blow-up index (\Cref{blow-up index}).
  The equality of (1) and (2) follows because the push-out \(L \to L(\beta)\) and pull-back \(K_C(-\beta) \to K_C\) induce the same map
  \[ \Ext^{1}(K_C, L) \to \Ext^1(K_C, L(\beta)) = \Ext^1(K_C(-\beta), L).\]

  We now prove the last statement.
  Note that the extension obtained by the pull-back along \(K_C(-\beta) \to K_C\) splits if and only if there is a map \(K_C(-\beta) \to \Omega_{\widetilde C}|_{C}\) such that the following diagram commutes
\[
\begin{tikzcd}
    K_C(-\beta) \arrow[d] \arrow[dr] & \\
    \Omega_{\widetilde{C}}|_C \arrow[r] & K_C 
\end{tikzcd}
\]
In this case, \(K_C(-\beta)\) is a sub-bundle of \(\Omega_{\widetilde C}|_C\) of degree \(2g-2-b\).
So \((3) \leq (2)\).
For the reverse inequality, let \(M \subset \Omega_{\widetilde C}|_{C}\) be a sub line-bundle of largest degree \(k\).
By \Cref{Using secant varieties to compute blow-up index} we know that the blow-up index of \(\widetilde C\) is at most \(\lceil 1/2(p_a+g-2)\rceil\).
So \(\Omega_{\widetilde C}|_C\) has a line sub-bundle of degree at least \(2g-2-\lceil 1/2(p_a+g-2)\rceil\).
The condition \(p_a > g+1\) ensures that
\[ 2g-2-\lceil 1/2(p_a+g-2)\rceil > \deg L = -p_a+2g-1.\]
Therefore \(\deg M = k> \deg L\), and so \(M\) does not admit a non-zero map to \(L\).
It follows that the composite \(M \to \Omega_{\widetilde C}|_C \to K_C\) is non-zero.
So \(M\) has the form \(K_C(-\beta)\) for some effective divisor \(\beta\) of degree \(2g-2-k\).
We conclude that \((2) \leq (3)\).
\end{proof}

As a corollary, we relate the blow-up index with the stability of the rank 2 bundle \(\Omega_{\widetilde C}|_{C}\).
\begin{corollary}
Let $\widetilde{C}$ be a ribbon of arithmetic genus $p_a > g+1$ on a smooth curve of genus $g$.
Then $\Omega_{\widetilde{C}}|_C$ is stable (resp. semistable) if and only if $b(\widetilde{C}) > 1/2(p_a+3)-2$ (resp. $b(\widetilde{C}) \geq 1/2(p_a+3)-2$).
\end{corollary}
\begin{proof}
  Note that \[\deg(\Omega_{\widetilde{C}}|_C) = -\deg(L)+2g-2 = -(p_a-2g+1)+2g-2 = -(p_a+3)+4g.\]
  So the slope of $\Omega_{\widetilde{C}}|_C$ is $-1/2(p_a+3)+2g$.
  The statement now follows from \Cref{equivalent definitions of blow-up index}.
\end{proof}
\begin{remark}
  For ribbons of blow-up index between $-1/2(p_a+3)-2$ and the maximum, which is $\lceil 1/2(p_a+g-2)\rceil$, the bundles $\Omega_{\widetilde C}|_{C}$ are semi-stable.
  But they are still distinguished by the maximum degree of sub line-bundles.
\end{remark}

\subsection{Blow-up index, gonality, and linear series Clifford index}

We end this section with bounds on the gonality of a ribbon in terms of the blow-up index.
Recall that the gonality of $\widetilde C$ is the minimum $d$ such that there exists a generalized $g^1_d$ on $\widetilde C$.

The following gives an upper bound on gonality.
\begin{proposition}\label{blow up index and linear series Cliff of a ribbon}
  Let $C$ be a smooth curve of genus $g$ and gonality $m$.
  Let $\widetilde{C}$ be a ribbon $C$ of arithmetic genus $p_a$, blow-up index $b$, and gonality $d$. Assume there exist a smooth divisor in $|-2L|$ and that
  $p_a > 2g-1+2m$, then
\begin{equation*}
     d \leq \min(b+2m, \lfloor 1/2(p_a+3) \rfloor).
\end{equation*}
\end{proposition}
\begin{proof}
  Let $\widetilde{C}'$ be the split ribbon obtained by blowing up $\widetilde{C}$ along a divisor of degree $b$.
  Then we have a degree 2 map $\widetilde{C'} \to C$.
  The pull-back of a $g^{1}_m$ on $C$ yields a $g^1_{2m}$ on $\widetilde C'$.
  Its push-forward to $\widetilde C$ gives a $g^1_{b+2m}$ on $\widetilde C$.
  So $d \leq b+2m$.

  Finally, in the prescribed range, all ribbons are smoothable \cite{Gon06}.
  A smooth curve of genus $p_a$ has gonality at most $\lfloor 1/2(p_a+3)\rfloor$.
  So the bound $d \leq \lfloor 1/2(p_a+3)\rfloor$ follows from semi-continuity (\Cref{semicontinuity of gonality}).
\end{proof}

The following gives a lower bound on gonality.
\begin{proposition}\label{lower bound to gonality}
  Let $C$ be a smooth curve of genus $g$ and gonality $m$.
  Let $\widetilde{C}$ be a ribbon $C$ of arithmetic genus $p_a$, blow-up index $b$, and gonality $d$.
  Then \[d \geq b-(2g-2).\]
\end{proposition}
\begin{proof}
  Let $\widetilde{M}$ denote the generalized line bundle associated to the $g^1_d$ on $\widetilde{C}$.
  We observe that $\widetilde{M}$ must be generated by the 2 sections of the $g^1_d$.
  If not, the subsheaf generated by these two sections gives a generalised $g^1_{d'}$ with $d' < d$, contradicting the minimality of $d$.
  
  Let us first treat the case that $\widetilde{M}$ is a line bundle.
  Since $\widetilde{M}$ is a line bundle, $d$ is even.
  Then the $g^1_d$ on $\widetilde C$ induces a finite map $\widetilde C \to \mathbb P^1$ of degree $d$, leading to the diagram
  \[
\begin{tikzcd}
    C \arrow[d, hook] \arrow[dr]{f} & \\
    \widetilde{C}  \arrow[r, ] & \mathbb{P}^1.
\end{tikzcd}
\]
The diagram above induces the diagram
\[
\begin{tikzcd}
    f^*(K_{\mathbb{P}^1}) \arrow[d, hook]\arrow[dr, hook] & \\
    \Omega_{\widetilde{C}}|_C \arrow[r] & K_C.
\end{tikzcd}
\]
Note that $f \colon C \to \mathbb P^1$ is finite of degree $d/2$.
We have the identification $f^{*}(K_{\mathbb P^1}) = K_C(-\beta)$, where $\beta$ is the ramification divisor of $f$, which has degree $2g-2+d$.
Owing to this diagram, the pull-back of the extension
\[ 0 \to L \to \Omega_{\widetilde C}|_{C} \to K_C \to 0\]
to $K_C(-\beta)$ is split.
Therefore, by \Cref{equivalent definitions of blow-up index}, we get $b \leq 2g-2+d$.

Let us treat the general case, where $\widetilde M$ is only a generalized line bundle.
Then there exists an effective divisor $\alpha$ of degree $a$ such that $\widetilde M$ is the push-forward of a line bundle $\widetilde M'$ on the blow-up $\widetilde C'$ of $\widetilde C$ at $a$.
Then we have a $g^1_{d-a}$ on $\widetilde C'$ with the line bundle $\widetilde M'$.
Plainly, the blow-up index $b'$ of $\widetilde C'$ is bounded below by $b-a$.
So we get
\[ d-a \geq b' - (2g-2) \geq b-a - (2g-2),\]
and hence $d \geq b-(2g-2)$, as required.
\end{proof}

\section{Gonality stratification of the space of ribbons}\label{sec:gonality}
In this section we describe the stratification of the space of ribbons by gonality.
Fix a smooth curve $C$ of genus $g$ and a line bundle $L$ on $C$ of negative degree.
Then the space of ribbons on $C$ with conormal bundle $L$ is naturally the projective space $\mathbb P H^0(K_C^{2} \otimes L^{-1})$.
Let \[W_d \subset \mathbb P H^0(K_C^{2} \otimes L^{-1})\] be the subset of ribbons that carry a generalized $g^1_d$.
Let \[\Sec_k(C) \subset \mathbb P H^0(K_C^{2} \otimes L^{-1})\] be the $k$-secant variety of $C \subset \mathbb P H^0(K_C^{2} \otimes L^{-1})$.

\begin{theorem}\label{gonality stratification}
  With the notation above, $W_d \subset \Sec_{d+2g-2}(C)$ is the union of linear subspaces spanned by divisors of degree $d+2g-2$ that contain the ramification divisor of a map $C \to \mathbb P^1$ of degree at most $d/2$.
\end{theorem}
\begin{proof}
  Let $\widetilde C \in W_{d}$.
  Then $\widetilde{C}$ carries a $g^1_d$, say $\widetilde M$.
  We show that $\widetilde C$ lies on a linear subspace spanned by a divisor of degree $d+2g-2$ that contains a ramification divisor of a map $C \to \mathbb P^1$ of degree at most $d/2$.
  It suffices to treat the case when the two sections of the $g^1_d$ genarate $\widetilde M$.
  (If not, we simply replace $\widetilde M$ by the subsheaf they generate).
  
  The generalised line bundle $\widetilde M$ is the push-forward of a line bundle $\widetilde M'$ on a blow-up $\widetilde C'$ of $\widetilde C$ in a divisor $\beta \subset C$ of degree $b$.
  Note that the degree of $\widetilde M'$ is $d-b$.
  The two sections of $\widetilde M'$ generate $\widetilde M'$ and hence yield a map $\widetilde C' \to \mathbb P^1$ of degree $d-b$.
  Its restriction $f \colon C \to \mathbb P^{1}$ is a map of degree $e = (d-b)/2$.
  Let $R$ be the ramification divisor of $f$.
  By \cite[Theorem~1.6]{BE95}, we have a map $K_C(-R) \to \Omega_{\widetilde C'}|_{C}$ that makes the following diagram commute
\begin{equation}\label{eqn:Rsplit}
\begin{tikzcd}
    & & & f^*(K_{\mathbb{P}^1}) = K_C(-R) \arrow[d]\arrow[dl] & \\
    0 \arrow[r] & L(\beta) \arrow[r] & \Omega_{\widetilde{C}'}|_C \arrow[r] & K_C \arrow[r] & 0.
  \end{tikzcd}
\end{equation}
By \Cref{equivalent definitions of blow-up index}, this means that $\widetilde C'$ splits after blowing it up in $R$.
Then $\widetilde C$ splits after blowing it up in $\beta + R$.
Note that $\beta + R$ has degree $d+2g-2$ and contains the ramification divisor $R$.

Conversely, let $R \subset C$ be the ramification divisor of a map $f \colon C \to \mathbb P^{1}$ of degree $e \leq d/2$, and let $\widetilde C$ lie in the linear span of a divisor of degree $d+2g-2$ of the form $\beta + R$, where $\beta$ is effective.
We produce a generalized $g^1_d$ on $\widetilde C$.
Since $\widetilde C$ lies in the span of $\beta + R$, the blow-up of $\widetilde C$ in $\beta + R$ is split (\Cref{blow-up index and secant variety}).
Let $\widetilde C'$ be the blow-up of $\widetilde C$ in $\beta$, so that the blow-up of $\widetilde C'$ in $R$ is split.
Then we have a map $K_C(-R) \to \Omega_{\widetilde C'}|_{C}$ making the diagram \eqref{eqn:Rsplit} commute.
Applying \cite[Theorem $1.6$]{BE95} again, we conclude that the map $f \colon C \to \mathbb P^1$ extends to a map $\widetilde C' \to \mathbb P^{1}$.
Then the pull-back of $\mathcal O(1)$ is a $g^1_{2e}$ on $\widetilde C'$, whose push-forward to $\widetilde C$ is a $g^1_d$ on $\widetilde C$.
\end{proof}

\begin{corollary}\label{inclusion of gonality stratification in secant variety}
  We have the inclusions
  \[ \Sec_{d-2m}(C) \subset W_d \subset \Sec_{d+2g-2}(C).\]
\end{corollary}
\begin{proof}
  The second inclusion is a part of \Cref{gonality stratification}.
  For the first, take $x \in \Sec_{d-2m}(C)$.
  Then $x$ is in the span of a divisor $\alpha$ of degree $d-2m$.
  Since $m$ is the gonality of $C$, we have a map $C \to \mathbb P^1$ of degree $m$.
  Its ramification divisor $R$ has degree $2g-2+2m$.
  Then $x$ is also in the span of $\alpha + R$.
\end{proof}

\Cref{gonality stratification} allows us to count the dimension of each gonality stratum.
To do so, let us further stratify $W_d$.
For a non-negative integer $e \leq d/2$, set $b = d-2e$.
Let $W_{e,b} \subset W_d$ be the union of linear subspaces spanned by divisors of degree $d+2g-2$ that contain the ramification divisor of a map $C \to \mathbb P^1$ of degree $e$.
Then, by definition, we have \[W_d = \bigcup_{e = 0}^{d/2} W_{e,b}.\]
From the proof of \Cref{gonality stratification}, we see that the points of $W_{e,b}$ correspond to ribbons $\widetilde C$ that carry a generalised $g^1_d$ that is the push-forward of a line bundle from a blow-up of $\widetilde C$ in a divisor of degree $b$.


Let $m$ be the gonality of $C$ and let $W^1_e(C)$ be the moduli space of base-point free linear series of rank $1$ and degree $e$ on $C$.
\begin{theorem}\label{dimension of W_d}
  Retain the setup above and assume that $L$ is a line bundle of degree $\leq -2m-5$.
  Set $p_a = 2g-1- \deg L$.
  Then 
  \[\dim(\overline W_{e,b}) \leq \dim W^1_e(C)+ 2g+2e+2b-3.\]
  In addition, if $C$ is a general curve of genus $g$, the following hold.
\begin{enumerate}
\item If $2e < g+2$, then $\overline W_{e,b}$ is empty, and otherwise, we have
  \[\dim(\overline W_{e,b}) \leq g+2d-5.\]
\item The gonality of a general ribbon on \(C\) with conormal bundle $L$ is the maximum, namely $\lfloor 1/2(p_a+3)\rfloor$.
  \end{enumerate}
\end{theorem}
\begin{proof}
  Consider a point \((\phi, \beta) \in W^1_e(C) \times \Sym^b(C)\).
  Let \(R_{\phi}\) be the ramification divisor of \(\phi\).
  Consider the span of \(R + \beta\) in \(\mathbf{P} H^0(K_C ^{2} \otimes L^{-1})\).
  For \((\phi, \beta)\) in a non-empty Zariski open subset \(U \subset W^1_e(C) \times \Sym^b(C)\), this span has a constant dimension \(n\).
  Since \(\deg R = 2g+2e-2\) and \(\deg \beta = b\), we have \(n \leq 2g+2e+b-3\).
  Consider the incidence variety
  \[ I \subset U \times \mathbb P H^0(K_C^{2} \otimes L^{-1})\]
  consisting of $(\phi, \beta, x)$ such that $x$ lies on the span of $R_{\phi} + \beta$.
  Then $I$ is a closed subvariety and its projection to $\mathbf{P} H^0(K_C^2 \otimes L^{-1})$ contains a Zariski open subset of $W_{e,b}$.
  The fibers of $I \to U$ are projective spaces of dimension $n$.
  Therefore, we get
  \begin{align*}
    \dim \overline W_{e,b}
    &\leq \dim I = \dim U + n = \dim W^1_e(C) + b + n\\
    &\leq \dim W^1_e(C) + b + (2g+2e+b-3),
  \end{align*}
  as required.

  If $C$ is general, then $\dim W^1_e(C) = 2e-g-2$ (and $W^1_e(C)$ is empty if $2e < g+2$).
  Since $d = 2e + b$, we get
  \[ \dim \overline W_{e,b} \leq g+2d-5.\]
  
  The dimension of the ambient $\dim \mathbf{P} H^0(K_c^2 \otimes L^{-1})$ is $g+p_a-3$.
  Therefore, if \[g+2d-5 < g+p_a-3,\] then $\overline W_{e,b} \subset \mathbf{P} H^0(K_c^2 \otimes L^{-1})$ is a proper subset.
  That is, a generic point of the ambient space is contained in $W_{d}$ only when $d \geq \lfloor (p_a+3)/2 \rfloor$.
  In other words, a generic ribbon has gonality at least $\lfloor (p_a+3)/2 \rfloor$.
  To see that a generic ribbon has gonality at most $\lfloor (p_a+3)/2 \rfloor$, observe that such a ribbon is a limit of smooth curves and smooth curves have gonality at most $\lfloor (p_a+3)/2 \rfloor$.
  Then apply the semi-continuity of gonality (\Cref{semicontinuity of gonality}).
\end{proof}

We end the section with some examples of the gonality loci of ribbons.
In all the examples, \(C\) will be a smooth curve of genus \(g\) and \(L\) a line bundle of negative degree on \(C\).
We set \[p_a = -\deg L + 2g - 1.\]
We consider \(C\) embedded in \(\mathbf{P}H^0(K_C^{\otimes 2} \otimes L^{-1})^{*}\) by the complete linear series.
Recall that the points of \(\mathbf{P}H^0(K_C^{\otimes 2} \otimes L^{-1})^{*}\) correspond to ribbons on \(C\) of arithmetic genus \(p_a\) and with conormal bundle \(L\).
\begin{example}[Ribbons on hyperelliptic curves]
  \label{ex:hyperelliptic}
  Let \(C\) be a hyperelliptic curve of genus \(g\).
  We describe \(W_{e,b}\) explicitly for \(e = 2\).
  
  Assume for simplicity that \(p_a \geq 2g+5\).
  Let \(R\) be the ramification divisor of the degree 2 map \(C \to \mathbf{P}^1\).
  Fix \(e = 2\) and \(d \geq 4\) and \(b = d-2e = d-4\).
  By \Cref{gonality stratification}, \(W_{2,d-4} \subset \mathbf{P}H^0(K_C^{\otimes 2} \otimes L^{-1})^{*}\) is the union of the span of \(R + \beta\) as \(\beta\) varies in \(\operatorname{Sym}^b(C)\).
  Let \(\langle R \rangle \subset \mathbf{P}H^0(K_C^{\otimes 2} \otimes L^{-1})^{*}\) be the span of \(R\).
  Then \(\langle R \rangle\) is a projective space of dimension \(2g+1\).
  We have the linear projection
  \[ \mathbf{P}H^0(K_C^{\otimes 2} \otimes L^{-1})^{*} \dashrightarrow \mathbf{P}H^0(K_C^{\otimes 2} \otimes L^{-1}(-R))^{*}\]
  with center \(\langle R \rangle\).
  The projection of \(C\) yields an embedding \(C \subset \mathbf{P}H^0(K_C^{\otimes 2} \otimes L^{-1}(-R))^{*}\) by the complete linear series.
  Then \(\overline W_{2,d-4} \subset \mathbf{P}H^0(K_C^{\otimes 2} \otimes L^{-1})^{*}\) is the cone over the \(b\)-secant variety of \(C\) in \(\mathbf{P}H^0(K_C^{\otimes 2} \otimes L^{-1}(-R))^{*}\).
  Note that the ambient space \(\mathbf{P}H^0(K_C^{\otimes 2} \otimes L^{-1}(-R))^{*}\) has dimension \(p_a-g-5\).
  Hence the \(b\)-secant variety has dimension
  \( \min(2b-1, p_a-g-5)\).
  Therefore, we get
  \[ \dim \overline W_{2,d-4} = \min(2d+2g-7, p_a+g-3).\]
  Note that \(2g+2g-7\) is the bound given by \Cref{dimension of W_d} and \(p_a+g-3\) is simply the dimension of the ambient projective space \(\mathbf{P}(K_C^{\otimes 2} \otimes L^{-1})^{*}\). 
  Observe that if \(d \geq (p_a-g+4)/2\), then \(\overline W_{2,d-4}\) is the ambient projective space.
  In particular, the gonality of a generic ribbon over a hyperelliptic curve is less than the maximum \(\lfloor(p_a+3)/2\rfloor\).
\end{example}

\begin{example}[Non-existence of a limiting \(g^1_3\)]
  \label{ex:ellipticg13}
  Let \(C\) be an elliptic curve and take \(p_a = 4\).
  Then \(C\) is a limit of smooth curves of genus 4, which have gonality 3.
  Nevertheless, from \Cref{gonality stratification} it follows that \(W_3\) is empty.
  So \(C\) does not admit a generalised \(g^1_3\).
  This does not contradict the semi-continuity of gonality (\Cref{semicontinuity of gonality}), since the condition \(p_a > d + 2g-1\) is not satisfied.
\end{example}

\begin{example}[An elliptic normal curve in \(\mathbf{P}^5\)]
  Let \(C\) be a smooth curve of genus \(1\) and take \(p_a = 7\).
  Then the conormal bundle \(L\) has degree \(-6\).
  The complete linear series gives an embedding \[C \subset \mathbf{P}(K_C^{\otimes 2} \otimes L^{-1})^{*} = \mathbf{P}^5.\]
  \Cref{gonality stratification} implies that \(W_{d}\) is empty for \(d = 1,2,3\).
  We describe \(W_4\).
  We have \(W_4 = W_{2,0}\) and it is the union of the spans of the ramification divisors of degree \(2\) maps \(C \to \mathbf{P}^1\).
  Fix one degree \(2\) map \(f \colon C \to \mathbf{P}^1\); all others are translates of \(f\).
  Let \(R = p_1 + p_2 + p_3 + p_4\) be the ramification divisor of \(f\) and \(P \subset \mathbf{P}^5\) the \(\mathbf{P}^3\) spanned by \(R\).
  Then \(W_4\) is the union of the translates of \(P\) by the points of the Jacobian of \(C\).
  This is a hypersurface of degree 6.
\end{example}

\section{ Green's conjecture and resolution Clifford index for a general ribbon on a general curve}\label{sec:green_general}

In this section, we relate the linear series Clifford index to the resolution Clifford index.
Specifically, we prove Green's conjecture for a general ribbon using Green's conjecture for a general smooth curve.
The proof follows the ideas in \cite{D18}.

We begin with two results about embedded and abstract deformations of ribbons.
\begin{proposition}\label{unobstructedness of embedded ribbons}
  Let $\widetilde{C}$ be a ribbon of arithmetic genus $p_a$ on a smooth curve $C$ of genus $g$.
  Suppose $p_a \geq 4g-2$.
  Let $\widetilde C \subset \mathbf{P}^N$ be the embedding by a complete linear series of a very ample line bundle of degree at least $2p_a$.
  Then the deformations of $\widetilde C \subset \mathbf{P}^N$ are unobstructed.
  That is, the Hilbert scheme of $\mathbf{P}^N$ is smooth at the point represented by $[\widetilde C]$.
\end{proposition}
\begin{proof}
  Let $L$ be the conormal bundle of the ribbon $\widetilde{C}$.

  We have the sequence
  \begin{equation}\label{eqn:c1}
    0 \to T_C \to T_{\mathbf{P}^N} \otimes \mathcal{O}_C \to N_{C/\mathbf{P}^N} \to 0
  \end{equation}
  and the Euler exact sequence
  \begin{equation}\label{eqn:c2}
    0 \to \mathcal{O}_C \to \mathcal{O}_C(1)^{\oplus N} \to  T_{\mathbf{P}^N} \otimes \mathcal{O}_C \to 0.
  \end{equation}
  Note that \(\deg \mathcal{O}_C(1) \geq p_a\) and \(\deg \mathcal{O}_C(1) \otimes L \geq 2g-1\).
  So both \(\mathcal{O}_C(1)\) and \(\mathcal{O}_C(1) \otimes L\) have vanishing \(H^1\).
  From the long exact sequence in cohomology applied to \eqref{eqn:c2} and its twist by \(L\), we get
  \begin{equation}\label{eqn:ct}
    H^{1}(T_{\mathbf{P}^N} \otimes \mathcal{O}_C) = 0 \text{ and } H^{1}(T_{\mathbf{P}^N} \otimes L) = 0.
  \end{equation}
  As a result, from the long exact sequence in cohomology applied to \eqref{eqn:c1} and its twist by \(L\), we get
  \begin{equation}\label{eqn:c}
    H^1(N_{C/\mathbf{P}^N}) = 0 \text{ and } H^1(N_{C/\mathbf{P}^N} \otimes L) = 0.
  \end{equation}

  We now turn to the normal bundle of \(\widetilde C\).
  Since $\widetilde{C}$ is a local complete intersection, its normal bundle is locally free of rank \(N-1\).
  Restricting it to \(C\) yields the exact sequence
   \begin{equation}\label{1}
    0 \to N_{\widetilde{C}/\mathbb{P}^N} \otimes L \to N_{\widetilde{C}/\mathbb{P}^N} \to N_{\widetilde{C}/\mathbb{P}^N} \otimes \mathcal{O}_C \to 0.
\end{equation}
We have the exact sequence
\[ 0 \to I_C^2/ (I_C \cdot I_{\widetilde C}) \to I_{\widetilde C} / (I_C \cdot I_{\widetilde C}) \to I_{\widetilde C}/I_C^2 \to 0,\]
whose terms are locally free \(\mathcal{O}_C\)-modules of ranks \(1\), \(N-1\), and \(N-2\), respectively.
Specifically, the kernel is the line bundle \(L^2\) and the middle term is the conormal bundle of \(\widetilde C\) restricted to \(C\).
Applying \(\Hom(-,\mathcal{O}_C)\) yields the sequence
\begin{equation}\label{2}
    0 \to \Hom(I_{\widetilde{C}}/I_{C}^2, \mathcal{O}_C) \to N_{\widetilde{C}/\mathbb{P}^N} \otimes \mathcal{O}_C \to L^{-2} \to 0 
\end{equation}
Finally, we have the sequence
\[ 0 \to I_{\widetilde C}/I_C^{2} \to I_C/I_C^2 \to I_C/I_{\widetilde C} \to 0,\]
whose terms are locally free \(\mathcal{O}_C\)-modules of ranks \(N-2\), \(N-1\), and \(1\), respectively.
Specifically, the cokernel is the line bundle \(L\) and the middle term is the conormal bundle of \(C\).
Applying \(\Hom(-,\mathcal{O}_C)\) yields the sequence
\begin{equation}\label{3}
    0 \to L^{-1} \to N_{C/\mathbb{P}^N} \to \Hom(I_{\widetilde{C}}/I_{C}^2, \mathcal{O}_C) \to 0.
\end{equation}
Using the vanishings \eqref{eqn:c}, the long exact sequence in cohomology applied to \eqref{3} and its twist by \(L\) yields
\begin{equation}\label{eqn:int}
  H^1(\Hom(I_{\widetilde C}/I_C^2, \mathcal{O}_C)) = 0 \text{ and } H^1(\Hom(I_{\widetilde C}/I_C^2, \mathcal{O}_C) \otimes L) = 0.
\end{equation}
Note that \(\deg L^{-2} \geq \deg L^{-1} = p_a-2g+1 \geq 2g-1\).
So \(H^1(L^{-2}) = H^1(L^{-1}) = 0\).
Combined with the vanishing \eqref{eqn:int}, the long exact sequence in cohomology applied to \eqref{2} and its twist by \(L\) yields
\begin{equation}\label{eqn:cc}
  H^1(N_{\widetilde C/\mathbf{P}^N} \otimes \mathcal{O}_C) = 0 \text{ and } H^1(N_{\widetilde C/\mathbf{P}^N} \otimes \mathcal{O}_C \otimes L) = 0.
\end{equation}
Finally, the long exact sequence in cohomology applied to \eqref{1} yields
\[ H^1(N_{\widetilde C/\mathbf{P}^N}) = 0.\]
So the embedded deformations of \(\widetilde C \subset \mathbf{P}^N\) are unobstructed.
\end{proof}

We now turn to the abstract deformations of \(\widetilde C\).
\begin{proposition}\label{unobstructedness of ribbons}
   Let $\widetilde{C}$ be a ribbon of arithmetic genus $p_a$ on a smooth curve $C$ of genus $g$.
   Suppose $p_a \geq 4g-2$.
   Then the deformations of $\widetilde{C}$ are unobstructed.
   That is, a versal deformation space of $\widetilde{C}$ is smooth.
 \end{proposition}
 \begin{proof}
   Choose an embedding $\widetilde C \subset \mathbf{P}^{N}$ by the complete linear series of a very ample line bundle of degree at least $2p_a$.
   We relate the abstract versus embedded deformations of $\widetilde C$.

   Consider the conormal sequence
\[0 \to \mathcal{I}/\mathcal{I}^2 \to \Omega_{\mathbb{P}^N}|_{\widetilde{C}} \to \Omega_{\widetilde{C}} \to 0.\]
Applying $\Hom(-, \mathcal{O}_{\widetilde{C}})$ yields
\begin{equation}\label{eqn:longext}
  \Hom(\mathcal{I}/\mathcal{I}^2, \mathcal{O}_{\widetilde{C}}) \to \Ext^1(\Omega_{\widetilde{C}}, \mathcal{O}_{\widetilde{C}}) \to \Ext^1(\Omega_{\mathbb{P}^N}|_{\widetilde{C}}, \mathcal{O}_{\widetilde{C}}) \to \Ext^1(\mathcal{I}/\mathcal{I}^2, \mathcal{O}_{\widetilde{C}}) \to \Ext^2(\Omega_{\widetilde{C}}, \mathcal{O}_{\widetilde{C}}).
\end{equation}
Consider $\Ext^1(\Omega_{\mathbb{P}^N}|_{\widetilde{C}},\mathcal{O}_{\widetilde{C}}) = H^1(T_{\mathbb{P}^N}|_{\widetilde{C}}).$
We have the exact sequence
\[ 0 \to T_{\mathbf{P}^N}|_C \otimes L \to T_{\mathbf{P}^N}|_{\widetilde C} \to T_{\mathbf{P}^N|_C} \to 0.\]
By \eqref{eqn:ct}, the \(H^1\) of the first and the third term vanishes.
So \(H^1\) of the middle term also vanishes.
Now, \eqref{eqn:longext} says that
\[ H^0(N_{\widetilde C/\mathbf{P}^N}) \to \Ext^1(\Omega_{\widetilde C}, \mathcal{O}_{\widetilde C})\]
is surjective, and
\[ H^1(N_{\widetilde C/\mathbf{P}^N}) \to \Ext^2(\Omega_{\widetilde C}, \mathcal{O}_{\widetilde C})\]
is injective.
That is, the forgetful map from embedded to abstract deformations is surjective on tangent spaces and injective on obstruction spaces.
As a result, the forgetful map is smooth.
We know that the embedded deformation space is smooth  (\Cref{unobstructedness of embedded ribbons}).
We conclude that the abstract deformation space is also smooth.
\end{proof}

We examine the non-nodal locus in a versal deformation.
\begin{proposition}\label{worse than nodal}
  Let $\widetilde{C}$ be a ribbon of arithmetic genus $p_a$  on a smooth curve of genus $g$ with $p_a \geq 3g-1$.
  Let \((U,0)\) be a versal deformation space of \(\widetilde C\) with a versal family of deformations $\mathcal{C} \to U$.
  Let $N \subset U$ be the closed subset containing $u \in U$ such that $\mathcal{C}_u$ has worse than nodal singularities.
  Then \(N \subset U\) has codimension at least $2$ at \(0\).
\end{proposition}
\begin{proof}
  We follow the proof of \cite[Proposition~3]{D18}.
  Since every ribbon $\widetilde{C}$ can be isotrivially deformed to a split ribbon, it is enough to prove the theorem for the split ribbon.
  So assume that $\widetilde{C}$ is the split ribbon.
  Let the conormal bundle be $L$ and let \(V = H^0(-2L)\).
  Every \(v \in V\) gives a double cover of \(C\) whose branch divisor is the zero locus of \(v\) and whose structure sheaf is \(\operatorname{O}_C \oplus L\).
  Consider the family of curves \(\mathcal C_V \to V\), whose fiber over \(v \in V\) is the double of \(C\) defined by \(v\).
  Note that the fiber over \(0\) is the split ribbon \(\mathcal C\).
  By versality, possibly after passing to an \'etale neighborhood of \(0\), we have a map \((V,0) \to (U,0)\) and an isomorphism of \(\mathcal C_V \to V\) with the pull-back of the versal family.
  Let \(N_V \subset V\) be the pre-image of \(N \subset U\).
  Then \(N_V \subset V\) is the locus of \(v \in V\) that define a worse than nodal \(\mathcal C_v\).
  But the double cover \(\mathcal C_v\) is has worse than nodal singularities if and only if \(v\) has a zero of multiplicity at least 3.
  Since \(\deg (-L) = p_a-2g+1 \geq g\), we have \(\deg(-2L) \geq 2g\).
  Then it follows that \(N_V \subset V\) has codimension 2 at \(0\).
  Therefore, \(N_U \subset U\) has codimension at least 2 at \(0\).
\end{proof}

As in \cite{D18}, let $\mathcal{U}$ denote the open sub-stack of the stack of all projective curves whose points corresond to curves $X$ satisfying the following conditions:
\begin{enumerate}
    \item $X$ is Gorenstein of arithmetic genus $p_a$ and $h^0(\mathcal{O}_{X}) = 1$.
    \item $K_{X}$ embeds $X$ as a arithmetically Gorenstein subscheme of $\mathbb{P}^{p_a-1}$.
    \item A versal deformation space of $X$ is smooth. 
    \end{enumerate}
    Let $M_{p_a}$ be the stack of smooth projective curves of genus $p_a$, and let $M_{p_a}^{nh} \subset M_{p_a}$ be the stack of non-hyperelliptic curves.

    Let $C$ be a smooth curve of genus $g$ and let $\widetilde C$ be a ribbon on $C$ of arithmetic genus $p_a$ with $p_a \geq 4g-2$.
    Then $X = \widetilde C$ satisfies the three conditions above: the first is clear; the second is \Cref{canonical morphism of ribbons}; and the third is \Cref{unobstructedness of ribbons}.
    Therefore, $\widetilde{C}$ defines a point $[\widetilde{C}]$ of $\mathcal{U}$.

    Let $p_a$ be odd, say $p_{a} = 2k+1$.
    As shown in \cite[Proposition~4]{D18}, we have a divisor $D \subset \mathcal U$ whose points represent curves $X$ such that $K_{k-1,1}(X, K_X) \neq 0$.
    We also have a divisor in $M_{p_a}^{nh}$ whose points represent smooth curves of gonality $k+1$.
    Let $D_{k+1} \subset \mathcal U$ be its closure.
    
\begin{proposition}\label{equality of the divisors}
  Let $\widetilde{C}$ be a ribbon of arithmetic genus $p_a = 2k+1$ on a smooth projective curve $C$ of genus $g$ with $p_a \geq 4g-2$.
  Then $D_{k+1} = D$ on an open subset of $\mathcal{U}$ containing $[\widetilde{C}]$.  
\end{proposition}
\begin{proof}
 The proof of \cite[Proposition~5]{D18} applies verbatim, thanks to \Cref{worse than nodal}.
\end{proof}

\begin{proposition}[Green's conjecture for generic ribbons of odd genus]
  \label{odd genus maximum blow-up index}
    Let $\widetilde{C}$ be a ribbon of arithmetic genus $p_a = 2k+1$ on a smooth projective curve $C$ of genus $g$ with $p_a \geq \operatorname{max}\{4g+2, 6g-4\}$.
    Suppose $\widetilde{C}$ has gonality $k+2$.
    Then its resolution Clifford index is $k$.
    That is, we have
    \[ \RCliff(\widetilde C) = \LCliff(\widetilde C) = k.\]
    In particular, the statement applies to a general ribbon of odd arithmetic genus 
    $p_a \geq \operatorname{max}\{3g+7, 6g-4\}$
    on a general curve of genus $g$.
\end{proposition}
\begin{proof}
  By the semi-continuity of gonality (\Cref{semicontinuity of gonality}), we see that $[\widetilde{C}] \notin D_{k+1}$.
  Therefore, by \Cref{equality of the divisors}, $[\widetilde{C}] \notin D$, so  $K_{k-1,1}(\widetilde{C}, K_{\widetilde{C}}) = 0$.
  That is, $\RCliff(\widetilde C) \geq k$.
  But the lower bound on $p_a$ further ensures that $\widetilde{C}$ is smoothable (see \cite{Gon06} or \cite{GGP08}).
  By semi-continuity, $\RCliff(\widetilde C)$ is at most $\RCliff$ of a general curve of genus $2k+1$; that is, $\RCliff(\widetilde C) \leq k$.
  So $\RCliff(\widetilde C) = k$.

  Since $\widetilde C$ is $(k+2)$-gonal, we have $\LCliff(\widetilde C) \leq k$.
  But we know that $\LCliff(\widetilde C) \geq \RCliff(\widetilde C)$ (\Cref{lin and res}).
  So $\LCliff(\widetilde C) = k$.

  The last assertion follows from \Cref{dimension of W_d}, $(2)$.
\end{proof}

The case of general ribbons of even arithmetic genus follows using blow-ups.
\begin{proposition}[Green's conjecture for generic ribbons of even genus]
  \label{generic green even}
  Let $C$ be a general smooth curve of genus $g$.
  Fix an even non-negative integer $p_a = 2k \geq \operatorname{max}\{3g+7, 6g-4\}$ and a line bundle $L$ of degree $-p_a+2g-1$ on $C$.
  Then a general ribbon $\widetilde C$ on $C$ with conormal bundle $L$ satisfies
  \[ \RCliff(\widetilde C) = \LCliff(\widetilde C) = k-1.\]
\end{proposition}
\begin{proof}
  We have $\RCliff(\widetilde C) \leq \LCliff(\widetilde C)$ by \Cref{lin and res}, and $\LCliff(\widetilde C) \leq k-1$ by \Cref{semicontinuity of gonality} (or \Cref{blow up index and linear series Cliff of a ribbon})  since $\widetilde{C}$ is smoothable under the condition on $p_a$. 
  It remains to prove that $k-1 \leq \RCliff(\widetilde C)$, that is,
  \[ K_{k,1}(\widetilde C, K_{\widetilde C}) = 0.\]
  By semi-continuity, it suffices to exhibit one $\widetilde C$ such that $K_{k,1}(\widetilde C, K_{\widetilde C}) = 0$.

  Choose a point $x \in C$ and let $\widetilde C'$ be a generic ribbon on $C$ with conormal bundle $L(-x)$.
  Let $\widetilde C$ be the blow-up of $\widetilde C'$  at $x$.
  Then $\widetilde C$ has conormal bundle $L$.
  Applying \Cref{odd genus maximum blow-up index} to $\widetilde C'$, we get
  \[ K_{k,1}(\widetilde C', K_{\widetilde C'}) = 0.\]
  By \cite[Lemma~1]{D18}, we deduce
  \[ K_{k,1}(\widetilde C, K_{\widetilde C}) = 0.\]
  The proof is now complete.
\end{proof}

The same idea allows us to obtain a lower bound on the resolution Clifford index of general ribbons of every blow-up index.
But we have to relinquish control over the conormal bundle.
\begin{proposition}\label{any genus and any blowup index}
  Let \(C\) be a general smooth curve of genus \(g\).
  Fix a non-negative integer $\operatorname{max}\{3g+7, 6g-4\}$ .
  Let \(\widetilde C\) be a general ribbon on \(C\) of arithmetic genus $p_a$ and blow-up index \(b\).
  Set $m = \lfloor (g+3)/2\rfloor$.
  Except in the case where $p_a$ is even and $g$ is odd and $b = \lceil (p_a+g-2)/2\rceil$ is the maximum possible, we have
 \[b-\lfloor g/2 \rfloor \leq \RCliff(\widetilde C) \leq \LCliff(\widetilde C) \leq \min(b+2m-2, \lfloor 1/2(p_a-1) \rfloor).\]
\end{proposition}
Note that the excluded case is handled by \Cref{generic green even}.
\begin{proof}
  Again, the second and the third inequalities follow from \Cref{lin and res} and \Cref{blow up index and linear series Cliff of a ribbon}.
  So we are left to prove the first inequality.
  By semi-continuity, it suffices to exhibit one ribbon $\widetilde C$ of arithmetic genus $p_a$ and blow-up index $b$ such that $b-\lfloor g/2 \rfloor \leq \RCliff(\widetilde C)$.

  The proof depends on the parity of $p_a$ and $g$.
  We fully explain the case when $p_a$ is even and $g$ is odd, and indicate the changes necessary in the other cases.
  Take
  \[ k = (p_a+g-2) - 2b.\]
  (If $g$ is even, add 1 to the right hand side.)
  Observe that $k \geq 0$ (in the excluded case, we have $k = -1$).
  Set $p_a' = p_a +k$, and observe that it is odd.
  Let $\widetilde C'$ be a general ribbon on $C$ of arithmetic genus $p_a'$.
  By \Cref{Using secant varieties to compute blow-up index}, the blow-up index of $\widetilde C'$ is
  \[ b' = (p'_a+g-2)/2 = b + k.\]
  Consider the sequence of $b'$ blow-ups that takes $\widetilde C'$ to the split ribbon.
  Let $\widetilde C$ be the ribbon obtained after $k$ blow-ups in this sequence.
  Then $\widetilde C$ has arithmetic genus $p_a$ and blow-up index $b$.

  By \Cref{odd genus maximum blow-up index} applied to $\widetilde C'$, we have
  \[ \RCliff(\widetilde C') = (p_a'-1)/2.\]
  That is,
  \[ K_{(p_a'-1)/2}(\widetilde C', K_{\widetilde C'}) = 0.\]
  By \cite[Lemma~1]{D18}, we conclude
  \[ K_{(p_a'-1)/2}(\widetilde C, K_{\widetilde C}) = 0,\]
  and therefore
  \begin{align*}
    \RCliff(\widetilde C) &\geq  (p_a-2) - (p_a'-1)/2 + 1\\
                          &= (p_a-k-1)/2 \\
    &= b-\lfloor g/2 \rfloor.
  \end{align*}
  The proof is now complete. 
\end{proof}

\section{Gonality and Clifford index of a split ribbon}\label{sec:split}
In this section we concentrate on the linear series Clifford index, gonality, and the resolution Clifford index of the split ribbon.

\subsection{Linear series Clifford index of a split ribbon}
We first find the linear series Clifford index.

\begin{proposition}\label{lower bounds to linear series gonality and Clifford index}
Let $\widetilde{C}$ be a ribbon of arithmetic genus $p_a$ on a smooth curve $C$ of genus $g$ and gonality $m$.
\begin{enumerate}
  \item If $\widetilde{C}$ has a $g^1_d$ then $d \geq 2m$.
  \item Suppose the Clifford index of \(C\) is \(m-2\).
    Then, if $\widetilde{C}$ has a $g^r_d$ with \(d \leq p_a-1\), then $d-2r \geq 2m-2$
\end{enumerate}

\end{proposition}
\begin{proof}
  We begin with some set-up for both statements.
  Let \(L\) be the conormal bundle of \(\widetilde C\).
  Let \((V, \widetilde M)\) be a \(g^r_d\) on \(\widetilde C\).
  Then there exists a divisor \(\beta \subset C\) such that \(\widetilde M\) is the push-forward of a line bundle on the blow-up of \(\widetilde C\) at \(\beta\).
  Let \(b\) be the degree of \(\beta\).
  Let \(\widetilde C'\) be the blow-up of \(\widetilde C\) at \(\beta\) and \(\widetilde M'\) such a line bundle on \(\widetilde C'\) whose push-forward is \(\widetilde M\).
  Note that \(\widetilde C'\) has arithmetic genus \(p_a-b\) and conormal bundle \(L(\beta)\).
  Set \(M = \widetilde M|_C / \text{torsion}\).
  On \(\widetilde C'\), we have the exact sequence
  \[ 0 \to M \otimes L (\beta) \to \widetilde M' \to M \to 0.\]
  Set \(e = \deg M\), so that \(\deg \widetilde M' = 2e = d -b \).

  Having set-up the basic objects, we turn to the proof of \((1)\).
  Assume that \(r = 1\).
  Since \((V, \widetilde M)\) is a generalised \(g^1_d\), the restriction map \(V \to H^0(M)\) is injective.
  Therefore, \(h^0(M) \geq 2\).
  Since \(C\) has gonality \(m\), we conclude that \(e = \deg M \geq m\).
  As a result, \(d = 2e + b \geq 2m\).
  
  We now assume $r = 2$ and turn to the proof of \((2)\).
  We make two cases:
  \begin{enumerate}
  \item Suppose \(e \geq 2g-2\).
    Then \(h^1(M) \leq 1\) and so \(h^0(M) \leq 2-g+e\).
    As a result, \(r \leq 1-g+e\), and
    \[ d-2r = 2e+b-2r \geq 2g-2+b \geq 2m-2.\]
  \item Suppose \(e < 2g-2\).
    Then \(M\) contributes to the linear series Clifford index of \(C\).
    That is, we have \(e - 2r \geq \LCliff(C)\).
    As a result, using the fact that \(\LCliff(C) \geq m-3\), (see \cite{MG91}) we get
    \[ d-2r = 2e + b - 2r \geq 2\LCliff(C)+2r+b \geq 2m-2.\]
  \end{enumerate}
\end{proof}

\begin{corollary}\label{linear series Cliff of a split ribbon}
  Let \(C\) be a smooth curve of gonality \(m\).
  Let \(\widetilde C\) be a split ribbon of arithmetic genus \(p_a\).
\begin{enumerate}
    \item The gonality of $\widetilde{C}$ is $2m$.
    \item The linear series Clifford index of $\widetilde{C}$ is \(2m-2\).
\end{enumerate}
\end{corollary}
\begin{proof}
  The pull-back of a \(g^1_m\) on \(C\) yields a \(g^1_{2m}\) on \(\widetilde C\).
  Therefore, the gonality of \(\widetilde C\) is at most \(2m\) and the linear series Clifford index is at most \(2m-2\).
  The lower bounds follow from \Cref{lower bounds to linear series gonality and Clifford index}.
\end{proof}

As a result of \Cref{linear series Cliff of a split ribbon}, we see that Green's conjecture for the split ribbon implies Green's conjecture for any smooth double cover of \(C\).
\begin{corollary}\label{green for split implies green for double cover}
  Let \(D\) be a smooth curve of genus \(p_a\) and \(f \colon D \to C\) a finite map of degree 2.
  Let \(L\) be the line bundle on \(C\) such that \(f_{*} \mathcal{O}_D \cong \mathcal{O}_C \oplus L\).
  Let \(\widetilde C\) be the split ribbon on \(C\) with conormal bundle \(L\).
  If \(\RCliff(\widetilde C) = \LCliff(\widetilde C)\), then \(\RCliff(D) = \LCliff(D)\).
\end{corollary}
\begin{proof}
  Let \(m\) be the gonality of \(C\).
  By \Cref{linear series Cliff of a split ribbon}, we have
  \[ \RCliff(\widetilde C) = 2m-2 = \gon(\widetilde C) - 2.\]
  Note that \(f \colon D \to C\) is specified by a global section of \(L^{-2}\).
  Scaling this section by \(t \in \mathbf{A}^1\) gives a family of double covers of \(C\) whose fiber at \(t = 0\) is \(\widetilde C\) and whose all other fibers are isomorphic to \(D\).
  By the semi-continuity of Koszul cohomology, we have,
  \[ \RCliff(\widetilde C) \leq \RCliff(D).\]
  Note that the pull-back of a \(g^1_{m}\) on \(C\) gives a \(g^{1}_{2m}\) on \(D\), so
  \[ \gon(D) \leq 2m,\]
  and hence
  \[ \LCliff(D) \leq 2m-2.\]
  The inequalities
  \[ 2m-2 = \RCliff(\widetilde C) \leq \RCliff(D) \leq \LCliff(D) \leq 2m-2\]
  yield the conclusion.
\end{proof}

\subsection{Resolution Clifford index of a split ribbon}
We now take up the resolution Clifford index of the split ribbon.
Unsurprisingly, we can describe the Koszul complex of the split ribbon entirely in terms of the underlying curve and the conormal bundle.

We begin by defining a graded module that governs the Koszul complex of the split ribbon.
Let \(C\) be a smooth curve of genus \(g\) and fix a line bundle \(L\) on \(C\).
Fix a non-negative integer \(p\).
Let \(S = \Sym(H^0(C, K_C))\) with the usual grading where the elements of \(H^0(C, K_C)\) have degree \(1\).
The graded \(S\)-module \(M^p_L\), or simply \(M^p\) if \(L\) is clear from the context, is defined by
\[ M^p = \bigoplus_{q \geq 0} K_{p,1}\left(C, K_C^{q}, K_C \otimes L^{-1}\right).\]
The \(S\)-module structure is induced by the multiplication maps
\[ K_{p,1}(C, K_C^{q}, K_C \otimes L^{-1}) \otimes H^0(C, K_C) \to K_{p,1}(C, K_C^{q+1}, K_C \otimes L^{-1}).\]

Consider the Koszul complex of the graded \(S\)-module \(M^p\), namely
\begin{equation}\label{eqn:mpkoszul}
 \cdots \to \bigwedge^{i+1} H^0(K_C) \otimes M^p_{q-1} \to \bigwedge^{i} H^0(K_C) \otimes M^p_q \to \bigwedge^{i-1}H^0(K_C) \otimes M^p_{q+1} \to \cdots.
\end{equation}
Let \(\Phi_{i,p,q}\) be the first map shown above.
Writing out the descriptions of the graded pieces of \(M^p\), we have
\begin{equation}\label{Phi}
  \Phi_{i,p,q} \colon \bigwedge^{i+1}H^0(K_C) \otimes K_{p,1}(C, K_C^{q-1}, K_C\otimes L^{-1}) \to \bigwedge^{i}H^0(K_C) \otimes K_{p,1}(C, K^q_C, K_C \otimes L^{-1}).
\end{equation}
Since \(M^p_{q}\)'s are themselves Koszul cohomology groups, the complex \eqref{eqn:mpkoszul} arises as a row in a double complex after taking vertical cohomology.
We now write this double complex.
For brevity, we write \((A)\) for \(H^0(C,A)\) and \(V^{(i)}\) for \(\bigwedge^i V\).
The double complex is
\begin{equation}\label{eqn:mpdouble}
  \begin{tikzcd}
    & \ar[d] & \ar[d] \\
   \ar[r]&  (K_C)^{(i+1)} \otimes (K_C \otimes L^{-1})^{(p+1)} \otimes (K_C^{q-1}) \ar[d] \ar[r] & (K_C)^{(i)} \otimes (K_C \otimes L^{-1})^{(p+1)} \otimes (K_C^{q}) \ar[d]\ar[r]& {}\\
 \ar[r]&  (K_C)^{(i+1)} \otimes (K_C \otimes L^{-1})^{(p)} \otimes (K_C^{q}\otimes L^{-1}) \ar[d] \ar[r,"d_{i,p,q}"] &  (K_C)^{(i)} \otimes (K_C \otimes L^{-1})^{(p)} \otimes (K_C^{q+1}\otimes L^{-1}) \ar[d]\ar[r]&{}\\
 \ar[r]&  (K_C)^{(i+1)} \otimes (K_C \otimes L^{-1})^{(p-1)} \otimes (K_C^{q+1}\otimes L^{-2})\ar[r]\ar[d] &   (K_C)^{(i)} \otimes (K_C \otimes L^{-1})^{(p-1)} \otimes (K_C^{q+2}\otimes L^{-2})\ar[r]\ar[d]&{}\\
     {} & {} &  {}
  \end{tikzcd}.
\end{equation}
The map \(\Phi_{i,p,q}\) in \eqref{Phi} is the map induced on vertical cohomology by the map \(d_{i,p,q}\).

\begin{lemma}\label{surjectivity as koszul cohomology}
  In the setup above, assume that \(h^1(-L) = 0\) and \(p \leq 2g-4\).
  Then \(\Phi_{i,p,1}\) is surjective if and only if \(K_{i,1}(M^p) = 0\).
\end{lemma}
\begin{proof}
  Note that $K_{i,1}(M^p)$ is the middle cohomology of the complex
  \[(K_C)^{(i+1)} \otimes K_{p,1}(C, K_C\otimes L^{-1}) \xrightarrow{\Phi_{i,p,1}} (K_C)^{(i)} \otimes K_{p,1}(C, K_C, K_C\otimes L^{-1}) \to (K_C)^{(i-1)} \otimes K_{p,1}(C, K^2_C, K_C\otimes L^{-1}).\]
  It suffices to prove that the last term vanishes.
  Let \(r + 1 = h^0(K_C \otimes L^{-1})\).
  By duality \cite[2.6.c]{G84}, we have
  \[  K_{p,1}(C, K^2_C, K_C\otimes L^{-1}) = K_{r-1-p,1}(C, K_C^{-1}, K_C \otimes L^{-1})^{*}.\]
  Shifting by 1 gives
  \[ K_{r-1-p,1}(C, K_C^{-1}, K_C \otimes L^{-1}) = K_{r-1-p,0}(C, L^{-1}, K_C \otimes L^{-1}).\]
  Since \(h^1(-L) = 0\), we get
  \[ h^0(-L) = h^0(K_{C}-L) - (2g-2).\]
  Since \(p \leq 2g-4\), the right hand side is bounded above by \(r-1-p\).
  By \cite[Theorem~3.a.1]{G84}, we conclude that \(K_{r-1-p,0}(C, L^{-1}, K_C \otimes L^{-1}) = 0\).
\end{proof}

The following theorem relates the syzygies of the canonically embedded split ribbon with the Koszul cohomology of the module \(M^p\).
\begin{theorem}\label{green for split}
  Let \(C\) be a smooth curve of genus \(g\) and gonality \(m\).
  Let \[p_a \geq \max(2g+2m-1, 6g-4)\]
  and let \(L\) be a line bundle on \(C\) of degree \(-p_a+2g-1\).
  Let \(\widetilde C\) be the split ribbon with conormal bundle \(L\).
  The following are equivalent.
  \begin{enumerate}
  \item \(\widetilde C\) satisfies Green's conjecture, that is,
    \[ \RCliff(\widetilde C) = \LCliff(\widetilde C).\]
  \item for all non-negative integers \(i,j\) such that \(i+j = 2m-3\), the map
    \[ \Phi_{i,j,1}: \bigwedge^{i+1}H^0(K_C) \otimes K_{j,1}(C, K_C-L) \xrightarrow{} \bigwedge^{i}H^0(K_C) \otimes K_{j,1}(C, K_C, K_C-L)\]
    is surjective.
  \item for all non-negative integers \(i,j\) such that \(i+j = 2m-3\), we have $K_{i,1}(M^{j}) = 0$.
\end{enumerate}
\end{theorem}

\begin{proof}
  Let \(\phi \colon \widetilde C \to C\) be the projection.
  For any line bundle \(A\) on \(C\), we have a canonically split exact sequence of \(\mathcal{O}_{\widetilde C}\)-modules
  \[ 0 \to A \otimes L \to \phi^{*} A \to A \to 0.\]
  The splitting \(A \to \phi^{*}A\) is given by the pull-back map.
  In particular, for the \(q\)-th tensor power of \(K_{\widetilde C} = \phi^{*}(K_C \otimes L^{-1})\), we have the splitting
  \[ K_{\widetilde C}^q = (K_C^q \otimes L^{-q}) \oplus (K_{C}^q \otimes L^{-q+1}).\]
  As a result, we have
  \begin{equation}\label{eqn:canring}
    H^0(\widetilde C, K_{\widetilde C}^q) = H^0(C,K_C^q \otimes L^{-q}) \oplus H^0(K_{C}^q \otimes L^{-q+1}).
  \end{equation}
  Let \(\widetilde S\) and \(S\) be the graded rings
  \begin{align*}
    \widetilde S &= \bigoplus_{q \geq 0} H^0(\widetilde C, K_{\widetilde C}^q), \text{ and }\\
    S &= \bigoplus_{q \geq 0} H^0(C, K_{C}^q \otimes L^{-q}),
  \end{align*}
  and let \(J\) be the graded \(S\)-module
  \[
    J = \bigoplus_{q \geq 0} H^0(C, K_C^q \otimes L^{-q+1}).
  \]
  Let \(\epsilon\) be a formal variable with \(\epsilon^2 = 0\).
  Taking the direct sum over all \(q\) in \eqref{eqn:canring} gives an isomorphism of graded rings
  \[ \widetilde S = S \oplus \epsilon J.\]

  The Koszul complex that computes \(K_{p,q}(\widetilde C, K_{\widetilde C})\) is
  \[\mathcal{K} = \cdots \to \bigwedge^{p+1} \widetilde S_1 \otimes \widetilde S_{q-1} \to \bigwedge^{p} \widetilde S_1 \otimes \widetilde S_{q} \to \bigwedge^{p-1} \widetilde S_1 \otimes \widetilde S_{q+1} \to \cdots \]
  Here, we take the three terms above to be in homological degres \(q-1, q, q+1\), respectively.
  The compelx \(\mathcal{K}\) has a subcomplex
  \[ \mathcal{S} = \cdots \to \bigwedge^{p+1} \widetilde S_1 \otimes \epsilon J_{q-1} \to \bigwedge^{p} \widetilde S_1 \otimes \epsilon J_{q} \to \bigwedge^{p-1} \widetilde S_1 \otimes \epsilon J_{q+1} \to \cdots,\]
  and the quotient complex is
  \[ \mathcal{Q} = \cdots \to \bigwedge^{p+1} \widetilde S_1 \otimes S_{q-1} \to \bigwedge^{p} \widetilde S_1 \otimes S_{q} \to \bigwedge^{p-1} \widetilde S_1 \otimes S_{q+1} \to \cdots.\]
  Observe that
  \[ \bigwedge^p \widetilde S_1 = \bigoplus_i \bigwedge^{i} S_1 \otimes \epsilon \bigwedge^{p-i} J_1.\]
  The action of \(\epsilon \bigwedge^{p-i} J_1\) on \(\epsilon J_q\) and on \(S_q\) is zero.
  As a result, we see that
  \begin{align*}
    H^q(\mathcal{S}) &= \bigoplus_i \bigwedge^{p-i} J_1 \otimes K_{i,q}(C, L, K_C \otimes L^{-1}), \text{ and }\\
    H^q(\mathcal{Q}) &= \bigoplus_i \bigwedge^{p-i} J_1 \otimes K_{i,q}(C, K_C \otimes L^{-1}).
  \end{align*}

  The short exact sequence of complexes
  \[ 0 \to \mathcal{S} \to \mathcal{K} \to \mathcal{Q} \to 0\]
  induces a long exact sequence in cohomology
  \begin{equation}\label{eq:les}
    \cdots \to H^{q-1}(\mathcal{Q}) \to H^q(\mathcal{S}) \to H^q(\mathcal{K}) \to \cdots.
  \end{equation}
  Consider the connecting map
  \[ \bigoplus_i \bigwedge^{p+1-i} J_1 \otimes K_{i,q-1}(C, K_C \otimes L^{-1}) \to \bigoplus_i \bigwedge^{p-i} J_1 \otimes K_{i,q}(C, L, K_C \otimes L^{-1}).\]
  It is easy to check that this map is diagonal, that is, the direct sum of maps
  \begin{equation}\label{eqn:conn}
    \bigwedge^{p+1-i} J_1 \otimes K_{i,q-1}(C, K_C \otimes L^{-1}) \to \bigwedge^{p-i} J_1 \otimes K_{i,q}(C, L, K_C \otimes L^{-1}).
  \end{equation}
  Note that \(J_{1} = H^0(C, K_C)\) and
  \[ K_{i,q}(C, L, K_C \otimes L^{-1}) = K_{i,q-1}(C, K_C, K_C \otimes L^{-1}),\]
  so the map \eqref{eqn:conn} becomes
  \begin{equation}\label{eq:conmap}
    \bigwedge^{p+1-i} H^0(K_C) \otimes K_{i,q-1}(C, K_C \otimes L^{-1}) \to \bigwedge^{p-i} H^0(K_C) \otimes K_{i,q-1}(C, K_C, K_C \otimes L^{-1}).
  \end{equation}
  
  For \(q = 2\), this is precisely the map \(\Phi_{p-i, i, 1}\) from \eqref{Phi}.

  We are now ready to prove the equivalence of (1) and (2).
  By \Cref{linear series Cliff of a split ribbon}, we know that \(\LCliff(\widetilde C) = 2m-2\).
  By \Cref{lin and res}, we know that \(\RCliff(\widetilde C) \leq \LCliff(\widetilde C)\).
  So, (1) is equivalent to the vanishing
  \[ K_{2m-3, 2}(\widetilde C, K_{\widetilde C}) = 0.\]
  Taking \(p = 2m-3\) and \(q = 2\), we have
  \[ K_{2m-3, 2}(\widetilde C, K_{\widetilde C}) = H^q(\mathcal{K}).\]
  Observe that we have
  \[ H^2(\mathcal{Q}) = \bigoplus_{i=0}^{2m-3} \bigwedge^{2m-3-i}H^0(K_C) \otimes K_{i,2}(C,K_C\otimes L^{-1}).\]
  By our hypothesis, we have
  \[ \deg(L^{-1}) = p_a - 2g+1 \geq 2m,\]
  and therefore \(\deg(K_C \otimes L^{-1}) \geq 2g+2m-2.\)
  By \cite[Theorem~4.a.1]{G84}, for \(0 \leq i \leq 2m-3\), we have
  \[K_{i,2}(C,K_C\otimes L^{-1}) = 0. \]
  So \(H^2(\mathcal{Q}) = 0\).
  By the long exact sequence \eqref{eq:les}, the vanishing of \(H^2(\mathcal{K})\) is equivalent to the surjectivity of
  \[ H^1(\mathcal{Q}) \to H^2(\mathcal{S}).\]
  By \eqref{eq:conmap}, this surjectivity is equivalent to (2).

  The equivalence of (2) and (3) is \Cref{surjectivity as koszul cohomology}.
\end{proof}


In light of \Cref{green for split}, it is natural to ask the following.
\begin{question}\label{split ribbon conjecture}
Are the equivalent statements in \Cref{green for split} are true ?   
\end{question}

We observe that they are true for hyperelliptic curves.

\begin{proposition}\label{split over hyperelliptic}
The conditions in \Cref{green for split} hold for a hyperelliptic curve $C$.
\end{proposition}
\begin{proof}
  We begin with a more general statement.
  Let \(C\) be a smooth curve of genus \(g\) and let \(m\) be its gonality.
  We show that the map
  \[\bigwedge^{i+1} H^0(K_C) \otimes K_{j,1}(C, K_C\otimes L^{-1}) \to \bigwedge^i H^0(K_C) \otimes K_{j,1}(C, K_C, K_C\otimes L^{-1})\]
  is surjective for \(i = 0\) and \(j = 2m-3\).
  To see this, we use the results and notation of \cite{But94}.
  For line bundles \(A\) and \(B\), the kernel of the map
  \[\bigwedge^{j}H^0(A) \otimes H^0(B) \to \bigwedge^{j-1} H^0(A) \otimes H^0(A \otimes B)\]
  is \(H^0(\bigwedge^j M_A \otimes A \otimes B)\) (see, for example, \cite[Lemma $1.4$]{EL93}).
  So it is enough to show that for $j = 2m-3$, the map
  \[H^0(K_C) \otimes H^0\left(\bigwedge^j M_{(K_C \otimes L^{-1})} \otimes K_C \otimes L^{-1}\right) \to H^0\left(\bigwedge^j M_{(K_C \otimes L^{-1})} \otimes K^2_C \otimes L^{-1}\right)\]
  is surjective.
  Since $K_C\otimes L^{-1}$ is semistable of degree (= slope) at least \(2g\), the bundle \(M_{K_C \otimes L^{-1}}\) is semistable (\cite[Theorem $1.2$]{But94}).
  The surjection follows from \cite[Proposition $2.2$]{But94}.
\end{proof}

We end with an example of a split ribbon that does \emph{not} satisfy Green's conjecture.
\begin{example}\label{does not satisfy Green}
  Let \(C\) be a general curve of genus 3, and take \(L = K_C^{-1}\).
  Note that \(C\) has gonality \(3\).
  Let \(\widetilde C\) be the split ribbon with conormal bundle \(L\).
  Then \(p_a(\widetilde C) = 9\).
  By \Cref{linear series Cliff of a split ribbon}, we have \(\LCliff(\widetilde C) = 4\).
  A \texttt{Macaulay2} computation---done over the base field \(\mathbf{Z}/101 \mathbf{Z}\) with a randomly chosen plane quartic \(C\)---shows that the minimal free resolution of the canonical embedding of \(\widetilde C\) has the betti table
  \[
    \begin{matrix}
 & 0 & 1 & 2 & 3 & 4 & 5 & 6 & 7\\
\text{total:} & 1 & 21 & 84 & 154 & 154 & 84 & 21 & 1\\
0: & 1 & . & . & . & . & . & . & .\\
1: & . & 21 & 64 & 90 & 64 & 20 & . & .\\
2: & . & . & 20 & 64 & 90 & 64 & 21 & .\\
3: & . & . & . & . & . & . & . & 1
    \end{matrix},\]
  which means \(\RCliff(\widetilde C) = 2\).
\end{example}

\color{black}

\bibliographystyle{plain}

\end{document}